\documentclass{amsart}

\usepackage[all]{xy}

\usepackage{pictexwd, amsthm, amssymb, amsmath, amsxtra, graphicx, psfrag, latexsym}

\usepackage{enumerate}

\DeclareMathOperator{\stab}{Stab}
\DeclareMathOperator{\z}{{\mathbb Z}}

\DeclareMathOperator{\g}{\Gamma}

\DeclareMathOperator{\im}{im}
\DeclareMathOperator{\rank}{rank}

\DeclareMathOperator{\f}{\mathcal F}
\DeclareMathOperator{\tf}{\widetilde {\mathcal F}}
\DeclareMathOperator{\vc}{\mathcal V\mathcal C}

\DeclareMathOperator{\fin}{\mathcal F\mathcal I\mathcal N}

\theoremstyle{plain}
\newtheorem{theorem}{Theorem}[section]

\newtheorem{corollary}[theorem]{Corollary}
\newtheorem{proposition}[theorem]{Proposition}

\theoremstyle{definition}

\begin{document}

\title[Lower algebraic K-theory of hyperbolic 3-simplex reflection groups.]
      {Lower algebraic K-theory of hyperbolic 3-simplex reflection groups.}
\author{Jean-Fran\c{c}ois\ Lafont}
\address{Department of Mathematics\\
         Ohio State University\\
         Columbus, OH  43210}
\email[Jean-Fran\c{c}ois\ Lafont]{jlafont@math.ohio-state.edu}
\author{Ivonne J.\ Ortiz}
\address{Department of Mathematics and Statistics\\
         Miami University\\
         Oxford, OH 45056}
\email[Ivonne J.\ Ortiz]{ortizi@muohio.edu}

\begin{abstract}
A hyperbolic 3-simplex reflection group is a Coxeter group arising
as a lattice in $O^+(3,1)$, with fundamental domain a geodesic
simplex in $\mathbb H^3$ (possibly with some ideal vertices).  The
classification of these groups is known, and there are exactly 9
cocompact examples, and 23 non-cocompact examples.  We provide a
complete computation of the lower algebraic K-theory of the integral
group ring of all the hyperbolic 3-simplex reflection groups.
\end{abstract}

\maketitle

\section[Introduction]{Introduction}

In this paper, we proceed to give a complete computation of the
lower algebraic $K$-theory of the integral group ring of all the
hyperbolic 3-simplex reflection groups.

We now proceed to outline the main steps of our approach.  Since the
groups $\g$ we are considering are lattices inside $O^+(3,1)$,
fundamental results of Farrell and Jones \cite{FJ93} imply that the
lower algebraic $K$-theory of the integral group ring $\mathbb Z \g$
can be computed by calculating $H_n^{\g}(E_{\vc}(\g); \mathbb
K\mathbb Z^{-\infty})$, a specific generalized equivariant homology
theory for a model for the classifying space $E_{\vc}(\g)$ of $\g$
with isotropy in the family $\mathcal {\vc}$ of virtually cyclic
subgroups of $\g$.

After introducing the groups we are interested in (see Section 2),
we then combine results from our previous paper \cite{LO} with a
recent construction of L\"uck and Weiermann \cite{LW} to obtain the
following explicit formula for the homology group above:
$$K_n(\z \Gamma)\cong H_n^{\Gamma}(E_{\fin}(\Gamma);\mathbb K\mathbb Z^{-\infty})
\oplus \bigoplus_{i=1}^k H_n^{V_i}(E_{\fin}(V_i)\rightarrow
*).$$
In the formula above, $E_{\fin}(\g)$ is a model for the classifying
space for proper actions, the collection $\{V_i\}_{i=1}^k$ are a
finite collection of virtually cyclic subgroups with specific
geometric properties, and $H_n^{V_i}(E_{\fin}(V_i)\rightarrow
*)$ are cokernels of certain relative assembly maps.  This explicit
formula is obtained in Section 3.

In view of this explicit formula, our computation reduces to being
able to:
\begin{enumerate}
\item identify for each of our groups the corresponding collection
$\{V_i\}$ of virtually cyclic subgroups (done in Section 4),
\item be able to calculate the cokernels of the corresponding
relative assembly maps (done in Section 6), and
\item calculate the homology groups
$H_n^{\Gamma}(E_{\fin}(\Gamma);\mathbb K\mathbb Z^{-\infty})$.
\end{enumerate}
For the computation of the homology groups, we note that Quinn
\cite{Qu82} has developed a spectral sequence for computing the
groups $H^{\g}_n(E_{\fin}(\g); \mathbb K\mathbb Z^{-\infty})$.  The
$E^2$-terms in the spectral sequence can be computed in terms of the
lower algebraic $K$-theory of the stabilizers of cells in a
$CW$-model for the classifying space $E_{\fin}(\g)$.

In Section 5, we proceed to give, for each of the finite subgroups
appearing as a cell stabilizer, a computation of the lower algebraic
$K$-theory.  We return to the spectral sequence computation in
Section 7, where we analyze some of the maps appearing in the
computation of the $E^2$-terms for the Quinn spectral sequence. In
all 32 cases, the spectral sequence collapses at the $E^2$ stage,
allowing us to complete the computations.  The reader who is merely
interested in knowing the results of the computations is invited to
consult Table 6 (for the uniform lattices) and Table 7 (for the
non-uniform lattices). Finally, in the Appendix, we provide a ``walk
through'' of the computations for two of the 32 groups we consider.

\vskip 20pt

\centerline{\bf Acknowledgments}

The authors would like to thank Tom Farrell, Ian Leary, and Marco
Varisco for many helpful comments on this project.  The graphics in
this paper were kindly produced by Dennis Burke.  The authors are
particularly grateful to Bruce Magurn for his extensive help with
the computations of the algebraic $K$-theory of finite groups
appearing in Section 5 of this paper.

The first author's work on this project was partially supported by NSF grant
DMS - 0606002.

\vskip 20pt

\section{The Three-dimensional groups}

A hyperbolic Coxeter $n$-simplex $\Delta^n$ is an $n$-dimensional
geodesic simplex in $\mathbb H^n$, all of  whose dihedral angles are
submultiples of $\pi$ or zero.  We allow  a simplex in $\mathbb H^n$
to be unbounded with ideal vertices on the sphere at infinity of
$\mathbb H^n$.  It is known that such simplices exist only in
dimensions  $n=2,3,\dots,9$, and that for $n \geq3$, there are
exactly 72 hyperbolic Coxeter simplices up to congruence, see
\cite{JKRT99} and \cite{JKRT02}.

A hyperbolic Coxeter $n$-simplex reflection group $\g$ is the group
generated by reflections in the sides of a Coxeter $n$-simplex in
hyperbolic $n$-space $\mathbb H^n$.  We will call such group a {\it
hyperbolic $n$-simplex group}.

According to Vinberg \cite{V67}, the associated hyperbolic
$n$-simplex groups of all but eight of the  72 simplices are
arithmetic. The nonarithmetic groups are the hyperbolic Coxeter
tetrahedra  groups $[(3,4,3,5)]$ $[5,3,6]$. $[5, 3^{[3]}]$,
$[(3^3,6)]$, $[(3,4,3,6)]$, $[(3,5,3,6)]$, and the 5-dimensional
hyperbolic Coxeter group $[(3^5,4)]$.

In dimension three, there are 32 hyperbolic Coxeter tetrahedra
groups. Nine of them are cocompact, see Figure 1, and 23 are
noncocompact, see Figure 2.  Let us briefly recall how the algebra
and geometry of these groups are encoded in the Coxeter diagrams.

From the algebraic viewpoint, the Coxeter diagram encodes a
presentation of the associated group $\Gamma$ as follows: associate
a generator $x_i$ to each vertex $v_i$ of the Coxeter diagram (hence
all of our groups will come equipped with four generators, as the
Coxeter diagrams have four vertices). For the relations in $\Gamma$,
one has:
\begin{enumerate}
\item for every vertex $v_i$, one inserts the relation $x_i^2=1$
\item if two vertices $v_i,v_j$ are not joined by an edge, one
inserts the relation $(x_ix_j)^2=1$ (so combined with the previous
relation, one sees that $x_i$ and $x_j$ commute, generating a $\z
/2\times \z /2$),
\item if two vertices $v_i,v_j$ are joined by an unlabelled edge,
one inserts the relation $(x_ix_j)^3=1$ (and in particular, the two
elements $x_i,x_j$ generate a subgroup isomorphic to the dihedral
group $D_3$),
\item if two vertices $v_i,v_j$ are joined by an edge with label $m_{ij}$,
one inserts the relation $(x_ix_j)^m_{ij}=1$ (and hence, the two
elements $x_i,x_j$ generate a subgroup isomorphic to the dihedral
group $D_{m_{ij}}$).
\end{enumerate}
A {\it special subgroup} of $\Gamma$ will be a subgroup generated by
a subset of the generating set.  Observe that such a subgroup will
automatically be a Coxeter group, with a presentation that can again
be read off from the Coxeter diagram.  Special subgroups generated
by a pair of generators will always be isomorphic to a (finite)
dihedral group.  An important point for our purposes is that in our
Coxeter groups, {\it every} finite subgroup can be conjugated into a
finite special subgroup.  In particular, since there are only
finitely many special subgroups, one can quite easily classify up to
isomorphism all the finite subgroups appearing in any of our 32
Coxeter groups.

Now let us move to the geometric viewpoint.  As we mentioned
earlier, associated to any of our 32 Coxeter groups, one has a
simplex $\Delta ^3$ in hyperbolic 3-space $\mathbb H^3$.  Each of
the four generators $x_i$ of the Coxeter group $\Gamma$ is
bijectively associated with the hyperplane $P_i$ extending one of
the four faces of the simplex $\Delta ^3$, and the angles between
the respective hyperplanes can again be read off from the Coxeter
diagram:
\begin{enumerate}
\item if two vertices $v_i,v_j$ are not joined by an edge, then
$\angle (P_i,P_j)=\pi/2$
\item if two vertices $v_i,v_j$ are joined by an unlabelled edge,
then $\angle (P_i,P_j) = \pi/3$
\item if two vertices $v_i,v_j$ are joined by an edge with label $m_{ij}$,
then $\angle (P_i,P_j) = \pi/ {m_{ij}}$.
\end{enumerate}
The resulting configuration of four hyperplanes exists, and is
unique up to isometries of $H^3$.  One can now define the map
$\Gamma \rightarrow O^+(3,1)=Isom(\mathbb H^3)$ by sending each
generator $x_i$ to the isometry obtained by reflecting in the
corresponding hyperplane $P_i$.  The condition on the angles between
the hyperplanes ensures that this map respects the relations in
$\Gamma$, and hence is actually a homomorphism.  In fact this map is
an embedding of $\Gamma$ as a discrete subgroup of $O^+(3,1)$, with
fundamental domain for the associated action on $\mathbb H^3$
consisting precisely of the simplex $\Delta ^3$.

Finally, to relate the geometric with the algebraic viewpoint, we
remind the reader of the following bijective identifications:
\begin{enumerate}
\item given an edge in the 3-simplex $\Delta ^3$, lying on the
intersection of two hyperplanes $P_i, P_j$, the subgroup of $\Gamma$
that fixes the edge pointwise is precisely the special subgroup
$\langle x_i, x_j\rangle$ (and hence will be a dihedral group).
\item given a vertex in the 3-simplex $\Delta ^3$, obtained as the
intersection of three hyperplanes $P_i, P_j, P_k$, the subgroup of
$\Gamma$ that stabilizes the vertex is precisely the special
subgroup $\langle x_i,x_j,x_k\rangle$.
\end{enumerate}
We point out that the stabilizer of a vertex of $\Delta ^3$ will
either be a finite Coxeter group (if the vertex lies inside $\mathbb
H^3$), or will be a 2-dimensional crystallographic group (if the
vertex is an ideal vertex).  Furthermore, one can readily determine
whether a vertex will be ideal or not, just by determining whether
the associated special subgroup is crystallographic or finite.

It is known that for all the groups listed above the Farrell and
Jones Isomorphism Conjecture in lower algebraic $K$-theory holds,
that is $\ H^{\g}_n(E_{\vc}(\g);\mathbb K\mathbb Z^{-\infty}) \cong
K_n(\mathbb Z\g)$ for $n<2$ (see \cite[Theorem 2.1]{Or04}). Our plan
is to use this result to explicitly compute the lower algebraic
$K$-theory of the integral group ring $\mathbb Z\g$, for all of the
32 groups listed above.


\hfill
\begin{figure}[h]
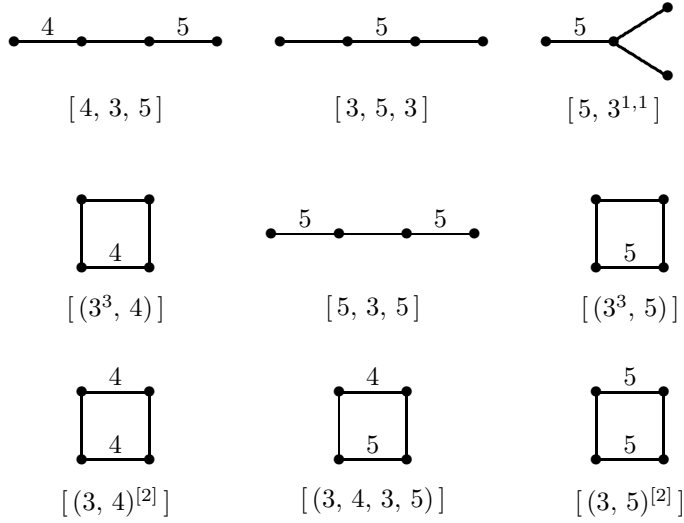

\hbox{
\vbox{\beginpicture
\setcoordinatesystem units <.9cm,.9cm> point at -.4 1
\setplotarea x from -.4 to 3.4, y from -1.4 to 1.4
\linethickness=.7pt
\putrule from 0 0  to 3 0
\put {$\bullet$} at  0  0
\put {$\bullet$} at  1 0
\put {$\bullet$} at  2  0
\put {$\bullet$} at 3  0
\put {$4$} [b] at   .5 .1
\put {$5$} [b] at   2.5 .1
\put {$[\,4,\,3,\,5\,]$} [t] at  1.5 -.8
\endpicture}

\vbox{\beginpicture
\setcoordinatesystem units <.9cm,.9cm> point at -.4 1
\setplotarea x from -.4 to 3.4, y from -1.4 to 1.4
\linethickness=.7pt
\putrule from 0 0  to 3 0
\put {$\bullet$} at  0  0
\put {$\bullet$} at  1 0
\put {$\bullet$} at  2  0
\put {$\bullet$} at 3  0
\put {$5$} [b] at   1.5 .1
\put {$[\,3,\,5,\,3\,]$} [t] at  1.5 -.8
\endpicture}

\vbox{\beginpicture
\setcoordinatesystem units <.9cm,.9cm> point at -.4 1
\setplotarea x from -.4 to 3.4, y from -1.4 to 1.4
\linethickness=.7pt
\putrule from 0 0 to 1 0
\setplotsymbol ({\circle*{.6}})
\plotsymbolspacing=.3pt        
\plot 1 0  1.8 .5 /
\plot 1 0 1.8 -.5 /
\put {$\bullet$} at  0 0
\put {$\bullet$} at  1 0
\put {$\bullet$} at 1.8 .5
\put {$\bullet$} at 1.8 -.5
\put {$5$} [b] at   .5 .1
\put {$[\,5,\,3^{1,1}\,]$} [t] at  1 -.8
\endpicture}
}

\hbox{
\hfil\vbox{\beginpicture
\setcoordinatesystem units <.9cm,.9cm> point at -.4 1
\setplotarea x from -.4 to 3.4, y from -1.4 to 1.4
\linethickness=.7pt
\putrule from 1 .5  to 2 .5
\putrule from 1 -.5  to 2 -.5
\putrule from 1 .5  to 1 -.5
\putrule from 2 .5  to 2 -.5
\put {$\bullet$} at  1  .5
\put {$\bullet$} at  1 -.5
\put {$\bullet$} at  2  .5
\put {$\bullet$} at  2  -.5
\put {$4$} [b] at   1.5 -.4
\put {$[\,(3^3,\,4)\,]$} [t] at  1.5 -.9
\endpicture}\hfil

\hfil\vbox{\beginpicture
\setcoordinatesystem units <.9cm,.9cm> point at -.4 1
\setplotarea x from -.4 to 3.4, y from -1.4 to 1.4
\linethickness=.7pt
\putrule from 0 0  to 3 0
\put {$\bullet$} at  0  0
\put {$\bullet$} at  1 0
\put {$\bullet$} at  2  0
\put {$\bullet$} at 3  0
\put {$5$} [b] at   .5 .1
\put {$5$} [b] at   2.5 .1
\put {$[\,5,\,3,\,5\,]$} [t] at  1.5 -.9
\endpicture}\hfil

\hfil\vbox{\beginpicture
\setcoordinatesystem units <.9cm,.9cm> point at -.4 1
\setplotarea x from -.4 to 3.4, y from -1.4 to 1.4
\linethickness=.7pt
\putrule from 1 .5  to 2 .5
\putrule from 1 -.5  to 2 -.5
\putrule from 1 .5  to 1 -.5
\putrule from 2 .5  to 2 -.5
\put {$\bullet$} at  1  .5
\put {$\bullet$} at  1 -.5
\put {$\bullet$} at  2  .5
\put {$\bullet$} at  2  -.5
\put {$5$} [b] at   1.5 -.4
\put {$[\,(3^3,\,5)\,]$} [t] at  1.5 -.9
\endpicture}\hfil
}

\hbox{
\hfil\vbox{\beginpicture
\setcoordinatesystem units <.9cm,.9cm> point at -.4 1
\setplotarea x from -.4 to 3.4, y from -1.4 to 1.4
\linethickness=.7pt
\putrule from 1 .5  to 2 .5
\putrule from 1 -.5  to 2 -.5
\putrule from 1 .5  to 1 -.5
\putrule from 2 .5  to 2 -.5
\put {$\bullet$} at  1  .5
\put {$\bullet$} at  1 -.5
\put {$\bullet$} at  2  .5
\put {$\bullet$} at  2  -.5
\put {$4$} [b] at 1.5 -.4
\put {$4$} [b] at 1.5  .6
\put {$[\,(3,\,4)^{[2]}\,]$} [t] at  1.5 -.9
\endpicture}\hfil

\hfil\vbox{\beginpicture
\setcoordinatesystem units <.9cm,.9cm> point at -.4 1
\setplotarea x from -.4 to 3.4, y from -1.4 to 1.4
\linethickness=.7pt
\putrule from 1 .5  to 2 .5
\putrule from 1 -.5  to 2 -.5
\putrule from 1 .5  to 1 -.5
\putrule from 2 .5  to 2 -.5
\put {$\bullet$} at  1  .5
\put {$\bullet$} at  1 -.5
\put {$\bullet$} at  2  .5
\put {$\bullet$} at  2  -.5
\put {$5$} [b] at   1.5 -.4
\put {$4$} [b] at 1.5  .6
\put {$[\,(3,\,4,\,3,\,5)\,]$} [t] at  1.5 -.9
\endpicture}\hfil

\hfil\vbox{\beginpicture
\setcoordinatesystem units <.9cm,.9cm> point at -.4 1
\setplotarea x from -.4 to 3.4, y from -1.4 to 1.4
\linethickness=.7pt
\putrule from 1 .5  to 2 .5
\putrule from 1 -.5  to 2 -.5
\putrule from 1 .5  to 1 -.5
\putrule from 2 .5  to 2 -.5
\put {$\bullet$} at  1  .5
\put {$\bullet$} at  1 -.5
\put {$\bullet$} at  2  .5
\put {$\bullet$} at  2  -.5
\put {$5$} [b] at   1.5 -.4
\put {$5$} [b] at 1.5  .6
\put {$[\,(3,\,5)^{[2]}\,]$} [t] at  1.5 -.9
\endpicture}\hfil
}
\caption{Cocompact hyperbolic Coxeter tetrahedral  groups}
\end{figure}
\hfil


\begin{figure}
\hbox{        
\hfil\vbox{\beginpicture
\setcoordinatesystem units <.75cm,.75cm> point at -.5 1
\setplotarea x from -.5 to 3.5, y from -1.4 to 1.4
\linethickness=.7pt
\putrule from 1 .5  to 2 .5
\putrule from 1 -.5  to 2 -.5
\putrule from 1 .5  to 1 -.5
\putrule from 2 .5  to 2 -.5
\put {$\bullet$} at  1  .5
\put {$\bullet$} at  1 -.5
\put {$\bullet$} at  2  .5
\put {$\bullet$} at  2  -.5
\put {$6$} [b] at   1.5 -.4
\put {$[\,(3^3,\,6)\,]$} [t] at  1.5 -.9
\endpicture}\hfil

\hfil\vbox{\beginpicture
\setcoordinatesystem units <.75cm,.75cm> point at -.5 1
\setplotarea x from -.5 to 3.5, y from -1.4 to 1.4
\linethickness=.7pt
\putrule from 1 .5  to 2 .5
\putrule from 1 -.5  to 2 -.5
\putrule from 1 .5  to 1 -.5
\putrule from 2 .5  to 2 -.5
\put {$\bullet$} at  1  .5
\put {$\bullet$} at  1 -.5
\put {$\bullet$} at  2  .5
\put {$\bullet$} at  2  -.5
\put {$4$} [b] at 1.5  .6
\put {$6$} [b] at 1.5 -.4
\put {$[\,(3,\,4,\,3,\,6)\,]$} [t] at  1.5 -.9
\endpicture}\hfil

\hfil\vbox{\beginpicture
\setcoordinatesystem units <.75cm,.75cm> point at -.5 1
\setplotarea x from -.5 to 3.5, y from -1.4 to 1.4
\linethickness=.7pt
\putrule from 1 .5  to 2 .5
\putrule from 1 -.5  to 2 -.5
\putrule from 1 .5  to 1 -.5
\putrule from 2 .5  to 2 -.5
\put {$\bullet$} at  1  .5
\put {$\bullet$} at  1 -.5
\put {$\bullet$} at  2  .5
\put {$\bullet$} at  2  -.5
\put {$5$} [b] at 1.5  .6
\put {$6$} [b] at   1.5 -.4
\put {$[\,(3,\,5,\,3,\,6)\,]$} [t] at  1.5 -.9
\endpicture}\hfil

\hfil\vbox{\beginpicture
\setcoordinatesystem units <.75cm,.75cm> point at -.5 1
\setplotarea x from -.5 to 3.5, y from -1.4 to 1.4
\linethickness=.7pt
\putrule from 1 .5  to 2 .5
\putrule from 1 -.5  to 2 -.5
\putrule from 1 .5  to 1 -.5
\putrule from 2 .5  to 2 -.5
\put {$\bullet$} at  1  .5
\put {$\bullet$} at  1 -.5
\put {$\bullet$} at  2  .5
\put {$\bullet$} at  2  -.5
\put {$6$} [b] at 1.5  .6
\put {$6$} [b] at   1.5 -.4
\put {$[\,(3,\,6)^{[2]}\,]$} [t] at  1.5 -.9
\endpicture}\hfil
}


\hbox{                 
\vbox{\beginpicture
\setcoordinatesystem units <.75cm,.75cm> point at -.5 1
\setplotarea x from -.5 to 3.5, y from -1.4 to 1.4
\linethickness=.7pt
\putrule from 0 0  to 3 0
\put {$\bullet$} at  0  0
\put {$\bullet$} at  1 0
\put {$\bullet$} at  2  0
\put {$\bullet$} at 3  0
\put {$5$} [b] at   .5 .1
\put {$6$} [b] at   2.5 .1
\put {$[\,5,\,3,\,6\,]$} [t] at  1.5 -.8
\endpicture}

\vbox{\beginpicture
\setcoordinatesystem units <.75cm,.75cm> point at -.5 1
\setplotarea x from -.5 to 3.5, y from -1.4 to 1.4
\linethickness=.7pt
\putrule from 0 0  to 3 0
\put {$\bullet$} at  0  0
\put {$\bullet$} at  1 0
\put {$\bullet$} at  2  0
\put {$\bullet$} at 3  0
\put {$6$} [b] at   .5 .1
\put {$6$} [b] at   2.5 .1
\put {$[\,6,\,3,\,6\,]$} [t] at  1.5 -.8
\endpicture}

\vbox{\beginpicture
\setcoordinatesystem units <.75cm,.75cm> point at -.5 1
\setplotarea x from -.5 to 3.5, y from -1.4 to 1.4
\linethickness=.7pt
\putrule from 0 0  to 3 0
\put {$\bullet$} at  0  0
\put {$\bullet$} at  1 0
\put {$\bullet$} at  2  0
\put {$\bullet$} at 3  0
\put {$6$} [b] at   2.5 .1
\put {$[\,3,\,3,\,6\,]$} [t] at  1.5 -.8
\endpicture}

\vbox{\beginpicture
\setcoordinatesystem units <.75cm,.75cm> point at -.5 1
\setplotarea x from -.5 to 3.5, y from -1.4 to 1.4
\linethickness=.7pt
\putrule from 0 0  to 3 0
\put {$\bullet$} at  0  0
\put {$\bullet$} at  1 0
\put {$\bullet$} at  2  0
\put {$\bullet$} at 3  0
\put {$4$} [b] at   .5 .1
\put {$6$} [b] at   2.5 .1
\put {$[\,4,\,3,\,6\,]$} [t] at  1.5 -.8
\endpicture}
}


\hbox{                     
\vbox{\beginpicture
\setcoordinatesystem units <.75cm,.75cm> point at -.5 1
\setplotarea x from -.5 to 3.5, y from -1.4 to 1.4
\linethickness=.7pt
\putrule from 0 0 to 1 0
\setplotsymbol ({\circle*{.5}})
\plotsymbolspacing=.3pt        
\plot 1 0  1.8 .5  1.8 -.5 /
\plot 1 0 1.8 -.5 /
\put {$\bullet$} at  0 0
\put {$\bullet$} at  1 0
\put {$\bullet$} at 1.8 .5
\put {$\bullet$} at 1.8 -.5
\put {$[\,3,\,3^{[3]}\,]$} [t] at  1 -.8
\endpicture}

\hfil\vbox{\beginpicture
\setcoordinatesystem units <.75cm,.75cm> point at -.5 1
\setplotarea x from -.5 to 3.5, y from -1.4 to 1.4
\linethickness=.7pt
\putrule from 0 0  to 3 0
\put {$\bullet$} at  0  0
\put {$\bullet$} at  1 0
\put {$\bullet$} at  2  0
\put {$\bullet$} at 3  0
\put {$6$} [b] at   1.5 .1
\put {$[\,3,\,6,\,3\,]$} [t] at  1.5 -.9
\endpicture}\hfil

\vbox{\beginpicture
\setcoordinatesystem units <.75cm,.75cm> point at -.5 1
\setplotarea x from -.5 to 3.5, y from -1.4 to 1.4
\linethickness=.7pt
\putrule from 0 0 to 1 0
\setplotsymbol ({\circle*{.5}})
\plotsymbolspacing=.3pt        
\plot 1 0  1.8 .5 /
\plot 1 0 1.8 -.5 /
\put {$\bullet$} at  0 0
\put {$\bullet$} at  1 0
\put {$\bullet$} at 1.8 .5
\put {$\bullet$} at 1.8 -.5
\put {$6$} [b] at   .5 .1
\put {$[\,6,\,3^{1,1}\,]$} [t] at  1 -.8
\endpicture}

\vbox{\beginpicture
\setcoordinatesystem units <.75cm,.75cm> point at -.5 1
\setplotarea x from -.5 to 3.5, y from -1.4 to 1.4
\linethickness=.7pt
\putrule from 0 0 to 1 0
\setplotsymbol ({\circle*{.5}})
\plotsymbolspacing=.3pt        
\plot 1 0  1.8 .5 1.8 -.5 /
\plot 1 0 1.8 -.5 /
\put {$\bullet$} at  0 0
\put {$\bullet$} at  1 0
\put {$\bullet$} at 1.8 .5
\put {$\bullet$} at 1.8 -.5
\put {$4$} [b] at   .5 .1
\put {$[\,4,\,3^{[3]}\,]$} [t] at  1 -.8
\endpicture}
}


\hbox{                      
\vbox{\beginpicture
\setcoordinatesystem units <.75cm,.75cm> point at -.5 1
\setplotarea x from -.5 to 3.5, y from -1.4 to 1.4
\linethickness=.7pt
\putrule from 0 0 to 1 0
\setplotsymbol ({\circle*{.5}})
\plotsymbolspacing=.3pt        
\plot 1 0  1.8 .5 1.8 -.5 /
\plot 1 0 1.8 -.5 /
\put {$\bullet$} at  0 0
\put {$\bullet$} at  1 0
\put {$\bullet$} at 1.8 .5
\put {$\bullet$} at 1.8 -.5
\put {$5$} [b] at   .5 .1
\put {$[\,5,\,3^{[3]}\,]$} [t] at  1 -.8
\endpicture}

\vbox{\beginpicture
\setcoordinatesystem units <.75cm,.75cm> point at -.5 1
\setplotarea x from -.5 to 3.5, y from -1.4 to 1.4
\linethickness=.7pt
\putrule from 0 0 to 1 0
\setplotsymbol ({\circle*{.5}})
\plotsymbolspacing=.3pt        
\plot 1 0  1.8 .5 1.8 -.5 /
\plot 1 0 1.8 -.5 /
\put {$\bullet$} at  0 0
\put {$\bullet$} at  1 0
\put {$\bullet$} at 1.8 .5
\put {$\bullet$} at 1.8 -.5
\put {$6$} [b] at   .5 .1
\put {$[\,6,\,3^{[3]}\,]$} [t] at  1 -.8
\endpicture}

\hfil\vbox{\beginpicture
\setcoordinatesystem units <.75cm,.75cm> point at -.5 1
\setplotarea x from -.5 to 3.5, y from -1.4 to 1.4
\linethickness=.7pt
\putrule from 1 .5  to 2 .5
\putrule from 1 -.5  to 2 -.5
\putrule from 1 .5  to 1 -.5
\putrule from 2 .5  to 2 -.5
\put {$\bullet$} at  1  .5
\put {$\bullet$} at  1 -.5
\put {$\bullet$} at  2  .5
\put {$\bullet$} at  2  -.5
\put {$4$} [b] at   1.5 -.4
\put {$4$} [l] at   2.1 0
\put {$[\,(3^2,\,4^2)\,]$} [t] at  1.5 -.9
\endpicture}\hfil

\hfil\vbox{\beginpicture
\setcoordinatesystem units <.75cm,.75cm> point at -.5 1
\setplotarea x from -.5 to 3.5, y from -1.4 to 1.4
\linethickness=.7pt
\putrule from 1 .5  to 2 .5
\putrule from 1 -.5  to 2 -.5
\putrule from 1 .5  to 1 -.5
\putrule from 2 .5  to 2 -.5
\put {$\bullet$} at  1  .5
\put {$\bullet$} at  1 -.5
\put {$\bullet$} at  2  .5
\put {$\bullet$} at  2  -.5
\put {$4$} [b] at   1.5 -.4
\put {$4$} [r] at   .9  0
\put {$4$} [l] at   2.1 0
\put {$[\,(3,\,4^3)\,]$} [t] at  1.5 -.9
\endpicture}\hfil
}


\hbox{                      
\hfil\vbox{\beginpicture
\setcoordinatesystem units <.75cm,.75cm> point at -.5 1
\setplotarea x from -.5 to 3.5, y from -1.4 to 1.4
\linethickness=.7pt
\putrule from 1 .5  to 2 .5
\putrule from 1 -.5  to 2 -.5
\putrule from 1 .5  to 1 -.5
\putrule from 2 .5  to 2 -.5
\put {$\bullet$} at  1  .5
\put {$\bullet$} at  1 -.5
\put {$\bullet$} at  2  .5
\put {$\bullet$} at  2  -.5
\put {$4$} [b] at   1.5 -.4
\put {$4$} [b] at   1.5 .6
\put {$4$} [r] at   .9  0
\put {$4$} [l] at   2.1 0
\put {$[\,4^{[4]}\,]$} [t] at  1.5 -.9
\endpicture}\hfil

\vbox{\beginpicture
\setcoordinatesystem units <.75cm,.75cm> point at -.5 1
\setplotarea x from -.5 to 3.5, y from -1.4 to 1.4
\linethickness=.7pt
\putrule from 0 0 to 1 0
\setplotsymbol ({\circle*{.5}})
\plotsymbolspacing=.3pt        
\plot 1 0  1.8 .5 /
\plot 1 0 1.8 -.5 /
\put {$\bullet$} at  0 0
\put {$\bullet$} at  1 0
\put {$\bullet$} at 1.8 .5
\put {$\bullet$} at 1.8 -.5
\put {$4$} [b] at   1.4 .4
\put {$4$} [t] at   1.4 -.4
\put {$[\,3,\,4^{1,1}\,]$} [t] at  1 -.8
\endpicture}

\vbox{\beginpicture
\setcoordinatesystem units <.75cm,.75cm> point at -.5 1
\setplotarea x from -.5 to 3.5, y from -1.4 to 1.4
\linethickness=.7pt
\putrule from 0 0  to 3 0
\put {$\bullet$} at  0  0
\put {$\bullet$} at  1 0
\put {$\bullet$} at  2  0
\put {$\bullet$} at 3  0
\put {$4$} [b] at   1.5 .1
\put {$4$} [b] at   2.5 .1
\put {$[\,3,\,4,\,4\,]$} [t] at  1.5 -.8
\endpicture}

\vbox{\beginpicture
\setcoordinatesystem units <.75cm,.75cm> point at -.5 1
\setplotarea x from -.5 to 3.5, y from -1.4 to 1.4
\linethickness=.7pt
\putrule from 0 0  to 3 0
\put {$\bullet$} at  0  0
\put {$\bullet$} at  1 0
\put {$\bullet$} at  2  0
\put {$\bullet$} at 3  0
\put {$4$} [b] at   .5 .1
\put {$4$} [b] at   1.5 .1
\put {$4$} [b] at   2.5 .1
\put {$[\,4,\,4,\,4\,]$} [t] at  1.5 -.8
\endpicture}
}

\hbox{                       
\vbox{\beginpicture
\setcoordinatesystem units <.75cm,.75cm> point at -.5 1
\setplotarea x from -.5 to 3.5, y from -1.4 to 1.4
\linethickness=.7pt
\putrule from 0 0 to 1 0
\setplotsymbol ({\circle*{.5}})
\plotsymbolspacing=.3pt        
\plot 1 0  1.8 .5 /
\plot 1 0 1.8 -.5 /
\put {$\bullet$} at  0 0
\put {$\bullet$} at  1 0
\put {$\bullet$} at 1.8 .5
\put {$\bullet$} at 1.8 -.5
\put {$4$} [b] at  .5 .1
\put {$4$} [b] at   1.4 .4
\put {$4$} [t] at   1.4 -.4
\put {$[\,4^{1,1,1}\,]$} [t] at  1 -1
\endpicture}

\vbox{\beginpicture
\setcoordinatesystem units <.75cm,.75cm> point at -.5 1
\setplotarea x from -.5 to 3.5, y from -1.4 to 1.4
\linethickness=.7pt
\setplotsymbol ({\circle*{.5}})
\plotsymbolspacing=.3pt        
\plot 0 0  1 .6  2 0  /
\plot 0 0  1 -.6 2 0 /
\putrule from 1 .6 to 1 -.6
\put {$\bullet$} at  0 0
\put {$\bullet$} at  2 0
\put {$\bullet$} at 1 .6
\put {$\bullet$} at 1 -.6
\put {$[\,3^{[3,3]}\,]$} [t] at  1 -1
\endpicture}

\vbox{\beginpicture
\setcoordinatesystem units <.75cm,.75cm> point at -.5 1
\setplotarea x from -.5 to 3.5, y from -1.4 to 1.4
\linethickness=.7pt
\setplotsymbol ({\circle*{.5}})
\plotsymbolspacing=.3pt        
\plot 0 -.3  1 .8  2 -.3  1 .1  0 -.3 /
\putrule from 1 .8 to 1 .1
\putrule from 0 -.3 to 2 -.3
\put {$\bullet$} at  0 -.3
\put {$\bullet$} at  2 -.3
\put {$\bullet$} at 1 .8
\put {$\bullet$} at 1 .1
\put {$[\,3^{[\ ]\times[\ ]}\,]$} [t] at  1 -1
\endpicture}
\hfill}

\caption{Noncocompact hyperbolic Coxeter tetrahedral groups}
\end{figure}


\section{A formula for the algebraic K-theory.}

In this section, we combine some recent work of L\"uck and
Weiermann \cite{LW} with some previous work of the authors \cite{LO}
to establish the following:

\begin{proposition}
Let $\f \subset \tf$ be a nested pair of families of subgroups of
$\Gamma$, and assume that the collection of subgroups
$\{H_\alpha\}_{\alpha \in I}$ is adapted to the pair $(\f, \tf)$.
Let $\mathcal H$ be a complete set of representatives of the
conjugacy classes within $\{H_\alpha\}$, and consider the cellular
$\Gamma$-pushouts:
$$\xymatrix{\coprod_{H\in \mathcal H} \Gamma \times_H E_{\f}(H)   \ar[d] _\alpha \ar[rrr]^\beta & & &
E_{\f}(\g) \ar[d]\\
\coprod_{H\in \mathcal H} \Gamma \times_H E_{\tf}(H) \ar[rrr] & & & X}
$$
Then $X$ is a model for $E_{\tf}(\Gamma)$.  In the above cellular
$\Gamma$-pushout, we require either (1) $\alpha$ is the disjoint union
of cellular $H$-maps ($H\in \mathcal H$), $\beta$ is an inclusion of
$\Gamma$-CW-complexes, or (2) $\alpha$ is the disjoint union
of inclusions of $H$-CW-complexes ($H\in \mathcal H$), $\beta$ is a
cellular $\Gamma$-map.
\end{proposition}

\begin{proof}
Let us start by recalling that a collection $\{H_{\alpha}\}_{\alpha
\in I}$ of subgroups of $\g$ is adapted to the pair $(\f, \tf)$
provided that:
\begin{enumerate}
\item For all $H_1, H_2 \in \{H_{\alpha}\}_{\alpha \in I}$,
either $H_1=H_2$, or $H_1 \cap H_2 \in \f$.
\item The collection $\{H_{\alpha}\}_{\alpha \in I}$ is {\it conjugacy
closed} i.e.\ if $H \in \{H_{\alpha}\}_{\alpha \in I}$ then
$gHg^{-1} \in \{H_{\alpha}\}_{\alpha \in I}$ for all $g \in \g$.
\item Every $H \in \{H_{\alpha}\}_{\alpha \in I}$ is {\it self-normalizing}, i.e.\ $N_{\Gamma}(H)=H$.
\item For all $G \in \tf \setminus \f$, there
exists $H \in \{H_{\alpha}\}_{\alpha \in I}$ such that $G \leq H$.
\end{enumerate}
Note that the subgroups in the collection $\{H_{\alpha}\}_{\alpha
\in I}$ are {\it not} assumed to lie within the family $\tf$.

Using the existence of the adapted family $\{H_{\alpha}\}_{\alpha
\in I}$, one can now define an equivalence relation on the subgroups
in $\tf - \f$ as follows: we decree that $G_1\sim G_2$ if there
exists an $H \in \{H_{\alpha}\}_{\alpha \in I}$ such that $G_1\leq
H$ and $G_2\leq H$.  Note that $\sim$ is indeed an equivalence
relation: the symmetric property is immediate, while reflexivity follows
from property (4) of adapted collection, and transitivity comes from
property (1) of adapted collection.  Furthermore this equivalence
relation has the
following two properties:
\begin{itemize}
\item if $G_1, G_2\in \tf -\f$ satisfies $G_1\leq G_2$, then $G_1\sim G_2$ (immediate
from the definition of $\sim$).
\item if $G_1, G_2\in \tf -\f$ and $g\in \Gamma$, then $G_1\sim G_2 \Leftrightarrow gG_1g^{-1}\sim
gG_2g^{-1}$ (follows from property (2) of adapted collection).
\end{itemize}
We denote by $[\tf - \f]$ the set of equivalence classes of elements
in $\tf -\f$ under the above equivalence relation, and for $G \in
\tf -\f$, we will write $[G]$ for the corresponding equivalence
class.  Note that by the second property above, the $\Gamma$-action
by conjugation on $\tf -\f$ preserves equivalence classes, and hence
descends to a $\Gamma$-action on $[\tf -\f]$.  We let $I$ be a
complete set of representatives $[G]$ of the $\Gamma$-orbits in
$[\tf -\f]$.  Finally, we define for $G\in \tf -\f$ the subgroup:
$$N_\Gamma[G]:= \{ g\in \Gamma \hskip 5pt | \hskip 5pt [gGg^{-1}] = [G]\}$$
which is precisely the isotropy group of $[G] \in [\tf -\f]$ under
the $\Gamma$-action induced by conjugation.  Finally, define a
family of subgroups $\tf[G]$ of the group $N_\Gamma[G]$ by:
$$\tf[G] := \{K \subset N_\Gamma[G] \hskip 5pt | \hskip 5pt K \in
\tf -\f, [K]=[G] \} \hskip 5pt \cup \hskip 5pt \{K \subset
N_\Gamma[G] \hskip 5pt | \hskip 5pt K \in \f\}$$ Observe that the
notions defined above (introduced in \cite{LW}) make sense for {\it
any} equivalence relation on $\tf - \f$ satisfying the two
properties above.

Now \cite[Theorem 2.3]{LW} states that for any equivalence relation
$\sim$ on the elements in $\tf-\f$ satisfying the two properties
above (and with the notation used in the previous paragraph), the
$\Gamma$-CW-complex $X$ defined by the cellular $\Gamma$-pushout
depicted below is a model for $E_{\tf}(\Gamma)$.

\vskip 5pt

$$\xymatrix{\coprod_{[H]\in I} \g \times_{N_{\g}[H]} E_{\f \cap N_{\g}[H]}(N_{\g}[H])   \ar[d]_\alpha
 \ar[rrr]^\beta & & &
E_{\f}(\g) \ar[d]\\
\coprod_{[H]\in I} \g \times_{N_{\g}[H]} E_{\tf [H]}(N_{\g}[H])
\ar[rrr] & & & X}
$$

\vskip 5pt

In the above cellular
$\Gamma$-pushout, L\"uck-Weiermann require either (1) $\alpha$ is the disjoint union
of cellular $N_\Gamma[H]$-maps ($[H]\in I$), $\beta$ is an inclusion of
$\Gamma$-CW-complexes, or (2) $\alpha$ is the disjoint union
of inclusions of $N_\Gamma[H]$-CW-complexes ($[H]\in I$), $\beta$ is a
cellular $\Gamma$-map.

We now proceed to verify that, for the equivalence relation we have
defined using the adapted family $\{H_{\alpha}\}_{\alpha \in I}$,
the left hand terms in the cellular $\Gamma$-pushout given above
reduce to precisely the left hand terms appearing in the statement
of our proposition.  This boils down to two claims:

\vskip 5pt

\noindent {\bf Claim 1:}  For any $G\in \tf -\f$, we have the
equality $N_\Gamma[G] = H$ where $H$ is the unique element in
$\{H_{\alpha}\}_{\alpha \in I}$ satisfying $G\leq H$.

\vskip 5pt

To see this, we first note that there indeed is a unique $H\in
\{H_{\alpha}\}_{\alpha \in I}$ satisfying $G\leq H$, for if there
were two such groups $H_1\neq H_2$, then we would immediately see
that $H_1\cap H_2 \geq G \in \tf -\f$, contradicting the property
(1) of an adapted collection.  Next we observe that if $h\in H$,
then $hGh^{-1} \leq hHh^{-1} = H$, and hence that $[hGh^{-1}]=[G]$,
which implies the containment $H \leq N_\Gamma[G]$.  Conversely, if
$k\in N_\Gamma[G]$, then we have that $[G]=[kGk^{-1}]$, and so from
the definition of the equivalence relation there must exist some
$\bar H \in \{H_{\alpha}\}_{\alpha \in I}$ with $G\leq \bar H$ and
$kGk^{-1} \leq \bar H$.  Since we already know that $G\leq H$, the
uniqueness forces $\bar H=H$, and thus that $kGk^{-1} \leq H$. This
in turn tells us that $H\cap k^{-1}Hk \geq G \in \tf -\f$, and
property (1) of an adapted collection now forces $H= k^{-1}Hk$,
which implies that $k\in N_\Gamma (H)$.  But property (3) of an
adapted collection forces the group $H$ to be self-normalizing,
giving $k\in H$, and completing the proof of the reverse inclusion.

\vskip 5pt

\noindent {\bf Claim 2:}  For any $G\in \tf -\f$, the family
$\tf[G]$ on the group $N_\Gamma[G]=H$ (see the previous Claim)
coincides with the restriction $\tf \cap H$ of the family $\tf$ to
the subgroup $H$ (i.e. consisting of all elements in $\tf$ that lie
within $H$).

\vskip 5pt

Note that the containment $\tf[G] \subset \tf \cap H$ is obvious
from the definition of $\tf[G]$.  For the opposite containment, let
$K\in \tf \cap H \subset \tf$, and observe that $K\leq H$ and either
$K\in \f$, or $K\in \tf -\f$.  In the first case, we have $K\in \f
\cap H \subset \tf [G]$, while in the second case, we have that
$[K]=[G]$ by the definition of the equivalence relation, and hence
again we have $K\in \tf[G]$.  This gives us the containment $\tf
\cap H \subset \tf[G]$, giving us the Claim.

\vskip 5pt

Having established our two Claims, we can now substitute the
expressions from the Claims for the corresponding ones in the
L\"uck-Weiermann diagram.  Finally, we comment on the indices in the
disjoint sums appearing in the right hand of the diagrams.  In the
expression of L\"uck-Weiermann, the disjoint sum is taken over $I$,
a complete system of representatives $[G]$ of the $\Gamma$-orbits in
$[\tf -\f]$.  But observe that from the definition of the
equivalence relation we are using, classes in $[\tf - \f]$ can be
bijectively identified with groups $H\in \{H_{\alpha}\}_{\alpha \in
I}$ (by associating each class in $[\tf - \f]$ with the unique
element in $\{H_{\alpha}\}_{\alpha \in I}$ containing all the
elements in the class).  Since it is clear that the $\Gamma$-action
on $[\tf -\f]$ coincides (under the bijection above) with the
$\Gamma$-action on the set $\{H_{\alpha}\}_{\alpha \in I}$, we can
replace the system of representatives $I$ by the system of
representatives $\mathcal H$.  This completes the proof of the
proposition.
\end{proof}

We now specialize to the case where $\f =\fin$ and $\tf = \vc$, and
recall that Bartels \cite{Bar03} has established that for any group
$\Gamma$, the relative assembly map:
$$H_*^{\Gamma}(E_{\fin}(\Gamma);\mathbb K\mathbb Z^{-\infty})
\rightarrow H_*^{\Gamma}(E_{\vc}(\Gamma);\mathbb K\mathbb
Z^{-\infty})$$ is split injective.  We denote by
$H_*^{\Gamma}(E_{\fin}(\Gamma)\rightarrow E_{\vc}(\Gamma))$ the
cokernel of the relative assembly map. Now applying the induction
structure on this equivariant generalized homology theory, an
immediate consequence of the previous proposition is the following:

\begin{corollary}
Given the group $\Gamma$, assume that the collection of subgroups
$\{H_\alpha\}_{\alpha \in I}$ is adapted to the pair $(\fin, \vc)$.
If $\mathcal H$ be a complete set of representatives of the
conjugacy classes within $\{H_\alpha\}$, then we have a splitting:
$$H_*^{\Gamma}(E_{\vc}(\Gamma);\mathbb K\mathbb Z^{-\infty})
\cong H_*^{\Gamma}(E_{\fin}(\Gamma);\mathbb K\mathbb Z^{-\infty})
\oplus \bigoplus_{H\in \mathcal H}H_*^{H}(E_{\fin}(H)\rightarrow
E_{\vc}(H)).$$
\end{corollary}

Next we recall that the authors established in \cite[Theorem
2.6]{LO} that in the case where $\Gamma$ is hyperbolic relative to a
collection of subgroups $\{H_i\}_{i=1}^k$ (assumed to be pairwise
non-conjugate), then the collection of subgroups consisting of:
\begin{enumerate}
\item All conjugates of $H_i$ (these will be called {\it peripheral subgroups}).
\item All maximal virtually infinite cyclic subgroups $V$ such that
$V \nsubseteq g H_i g^{-1}$, for all $i=1, \cdots k$, and for all $g
\in \g$.
\end{enumerate}
is adapted to the pair of families $(\fin, \vc)$.  Applying the
previous corollary to this special case, we get:

\begin{corollary}
Assume that the group $\Gamma$ is hyperbolic relative to the
collection of subgroups $\{H_i\}_{i=1}^k$ (assumed to be pairwise
nonconjugate). Let $\mathcal V$ be a complete set of representatives
of the conjugacy classes of maximal virtually infinite cyclic
subgroups inside $\Gamma$ which {\it cannot} be conjugated within
any of the $H_i$.  Then we have a splitting:
$$H_*^{\Gamma}(E_{\vc}(\Gamma);\mathbb K\mathbb Z^{-\infty})
\cong H_*^{\Gamma}(E_{\fin}(\Gamma);\mathbb K\mathbb Z^{-\infty})
\oplus \bigoplus_{i=1}^k H_*^{H_i}(E_{\fin}(H_i)\rightarrow
E_{\vc}(H_i))$$
$$\oplus \bigoplus _{V\in \mathcal V}H_*^{V}(E_{\fin}(V)\rightarrow
*).$$
\end{corollary}

The primary example of relatively hyperbolic groups are groups
$\Gamma$ acting with cofinite volume (but not cocompactly) on a
simply connected Reimannian manifold whose sectional curvature
satisfies $-b^2 \leq K \leq -a^2 <0$.  These groups are hyperbolic
relative to the ``cusp groups'', which one can take to be the
subgroups arising as stabilizers of ideal points in the boundary at
infinity of the Riemannian manifold.  We note that non-uniform
lattices in $O^+(n,1) = Isom(\mathbb H^n)$ are examples of
relatively hyperbolic groups, and for this class of groups, the cusp
groups are automatically $(n-1)$-dimensional crystallographic groups
(this is due to the fact that the horospheres have intrinsic
geometry isometric to $\mathbb R^{n-1}$). Observe that 23 of the
groups we are considering (see Figure 2) are non-uniform lattices in
$O^+(3,1)$, and hence are relatively hyperbolic groups, relative to
a collection of subgroups, each of which is isomorphic to a
$2$-dimensional crystallographic group.  In the situation of the
groups we are considering, the situation is even further simplified
by the following observations:
\begin{itemize}
\item Pearson \cite{Pe98} showed that for {\it any} $2$-dimensional
crystallographic group $H$, the relative assembly map is an
isomorphism for $n\leq 1$, and hence that:
$$H_n^{H}(E_{\fin}(H)\rightarrow E_{\vc}(H)) = 0$$
for $n\leq 1$.
\item the authors in \cite[Section 3]{LO} gave a general procedure
for classifying the maximal virtually cyclic subgroups of Coxeter
groups acting on $\mathbb H^3$.  The groups fall into three types,
with infinitely many conjugacy classes of type II and type III, and
only {\it finitely many} conjugacy classes of type I subgroups.
Furthermore, the relative assembly map is an isomorphism (for $n\leq
1$) for all groups of type II and III.  This is discussed in more
detail in Section 4.
\item Berkove, Farrell, Juan-Pineda, and Pearson \cite{BFPP00}
established that the Farrell-Jones isomorphism conjecture holds for
all lattices $\Gamma$ in hyperbolic space (and $k\leq 1$), and hence
that one has isomorphisms:
$$K_n(\z \Gamma)\cong H_n^{\Gamma}(E_{\vc}(\Gamma);\mathbb K\mathbb Z^{-\infty})$$
for all $n\leq 1$.
\end{itemize}
Combining these observations with the previous corollary yields the
following:

\begin{corollary}
Let $\Gamma \leq O^+(3,1)$ be any Coxeter group arising as a lattice
(uniform or non-uniform), and let $\{V_i\}_{i=1}^k$ be a complete
set of representatives for conjugacy classes of type I maximal
virtually infinite cyclic subgroups of $\Gamma$.  Then we have, for
all $n\leq 1$, isomorphisms:
$$K_n(\z \Gamma)\cong H_n^{\Gamma}(E_{\fin}(\Gamma);\mathbb K\mathbb Z^{-\infty})
\oplus \bigoplus_{i=1}^k H_n^{V_i}(E_{\fin}(V_i)\rightarrow
*).$$
\end{corollary}

In the next sections, we will implement this corollary to compute
the lower algebraic $K$-theory of the integral group rings of all 32
of the $3$-simplex hyperbolic reflection groups.

\section{Maximal infinite $\vc_{\infty}$ subgroups.}

In this section, we proceed to classify the maximal infinite
virtually cyclic ($\vc_{\infty}$) subgroups arising in our groups.
Let us start by briefly recalling some of the results from Section 3
of \cite{LO}. First of all, for a lattice $\g$ in $O^+(n,1)$,
infinite $\vc$ subgroups are of two types: those that fix a single
point in the boundary at infinity, and those that fix a pair of
points in the boundary at infinity.  We call subgroups of the first
type {\it parabolic}, and those of second type {\it hyperbolic}.
Note that every parabolic subgroup can be conjugated into a cusp
group; for the purpose of our classification, we will ignore these
subgroups. The subgroups of hyperbolic type automatically stabilize
the geodesic joining the pair of fixed points in the boundary at
infinity.  Furthermore the geodesic they stabilize will project to a
periodic curve in the quotient space $\mathbb H^n/\g$.  Note that
conversely, stabilizers of periodic geodesics are infinite $\vc$
subgroups of $\g$.  This implies that the {\it maximal} hyperbolic
type infinite $\vc$ subgroups of $\g$ are in bijective
correspondence with stabilizers of periodic geodesics.

We now specialize to the case where $n=3$, and $\g$ is a Coxeter
group.  In this situation, we can subdivide the family of periodic
geodesics into three types:
\begin{itemize}
\item a geodesic whose projection has non-trivial intersection with
the interior of the polyhedron $\mathbb H^3/\g$, which we call
type III.
\item a geodesic whose projection lies in the boundary of the
polyhedron $\mathbb H^3/\g$, but does not lie inside the
1-skeleton of $\mathbb H^3/\g$, which we call type II.
\item a geodesic whose projection lies in the 1-skeleton of the
polyhedron $\mathbb H^3/\g$, which we call type I.
\end{itemize}
For geodesics of type III, it is easy to see that the stabilizer of
the geodesic must be isomorphic to either $\z$ or $D_\infty$. For
geodesics of type II, the stabilizer is always isomorphic to either
$\z_2\times \z$ or $\z_2\times D_\infty$.  The main purpose of this
section will be to classify stabilizers of type I geodesics for all
32 groups which occur as hyperbolic 3-simplex reflection groups.

We start by outlining our approach: all the groups we are
considering have fundamental domain consisting of a 3-dimensional
simplex in $\mathbb H^3$ (possibly with some ideal vertices).  So up
to conjugacy, for each of the groups we are considering, we can have
{\it at most six} distinct stabilizers of type I geodesic (one for
each edge in the fundamental domain, fewer in the presence of ideal
vertices). But this is actually an overcount, as one could
potentially have a type I geodesic whose projection into the
fundamental domain passes through several of the edges.  So the
first step is to understand {\it how many} distinct stabilizers (up
to conjugacy) one obtains.

Let us explain how one can find out the number of distinct
stabilizers.  Note that, at every (non-ideal) vertex $v$ of our
fundamental domain 3-simplex in $\mathbb H^3$, we can consider a
small $\epsilon$-sphere $S_v$ centered at $v$.  Now the tessellation
of $\mathbb H^3$ by copies of the fundamental domain induces a
tessellation of $S_v$ by isometric spherical triangles.  In fact,
the tessellation of $S_v$ is the one naturally associated with the
special subgroup of the Coxeter group $\g$ that stabilizes the
vertex $v$. Now note that, if we were to label the three edges of
the 3-simplex incident to $v$, we get a corresponding label of the
three vertices of a spherical triangle in the tessellation of $S_v$.
One can extend this labeling via reflections, both for the
tessellation of $\mathbb H^3$ and the tessellation of $S_v$.

Now given a periodic geodesic of type I, with a portion of the
geodesic projecting to the edge $e$ in the 3-simplex, with $e$
adjacent to the vertex $v$, one can easily ``read off'' from the
labeled tessellation of $S_v$ which edge extends the geodesic.
Indeed, this will be picked up by the label of the vertex in the
tessellation of $S_v$ which is antipodal to the labeled vertex
corresponding to $e$.  In this manner, one can easily decide the
number of distinct stabilizers of type I geodesics that arise for
the 32 groups we are considering.

To recognize the tessellations arising for the various $S_v$, one
now notes that the isometry group of each of these tessellations can
be obtained by looking at the stabilizer $\g _v$ of the vertex $v$
in the group $\g$.  These stabilizers are finite special subgroups,
of the Coxeter group $\g$, generated by three of the four canonical
generators of $\g$. From the classification of the 32 hyperbolic
3-simplex groups, it is easy to list out all the finite parabolic
subgroups we need to consider: there are eight of these, namely
$\z_2\times D_2$, $\z_2\times D_3$, $\z_2\times D_4$, $\z_2\times
D_5$, $\z_2\times D_6$, $[3,3]\cong S_4$, $[3,4]\cong \z_2 \times
S_4$, and $[3,5] \cong \z_2 \times A_5$.

The next step is to identify the stabilizers of the corresponding
geodesics.  In the situation we are considering, all the type I
geodesics $\eta$ that appear have stabilizer $\stab_{\g}(\eta)$
acting with fundamental domain an interval. In fact, the interval
can be identified with the quotient space $\eta/\stab_{\g}(\eta)
\subset \mathbb H^3/\g$, which will be a union of edges in the
1-skeleton of the 3-simplex $\mathbb H^3/\g$.  Hence the group
$\stab_{\g}(\eta)$ can be identified using Bass-Serre theory: it
will be the fundamental group of a graph of group, where the graph
of group consists of a single edge joining two vertices, with
edge/vertex groups which can be explicitly found from the
tessellations.  We will say that the geodesic (or sometimes the edge
in the 1-skeleton) {\it reflects} at the two endpoint vertices.

Indeed, the edge group $G_e$ will be precisely the stabilizer of one
of the edges in $\eta/\stab_{\g}(\eta) \subset \mathbb
H^3/\g$. On the other hand, the vertex groups $G_v,G_w$ can be
found by looking at each of the two endpoint vertices $v,w$ for
$\eta/\stab_{\g}(\eta)$, and studying the spherical tessellations
of $S_v,S_w$.  Note that we are trying to identify elements in
$\g$, which stabilize the vertex $v$ (respectively $w$), and
{\it additionally} map the geodesic $\eta$ through $v$ to itself. In
particular, it must map the pair of antipodal vertices $\eta^{\pm}$
(corresponding to the incoming/outgoing $\eta$-directions) in the
tessellation of $S_v$ to themselves.  The subgroup $G_e\subset G_v$
can be identified with the index 2 subgroup consisting of elements
$G_v$ which fix both of the points $\eta ^{\pm}$.  Now there is an
obvious map which permutes the two points $\eta ^+$ and $\eta ^-$:
namely the reflection in the equator equidistant from these two
points.  But it is not clear that this reflection preserves the
tessellation of $S_v$; in some cases, one will need to reflect in
the equator, and then rotate by a certain angle along the $\eta^{\pm}$
axis, in order to obtain an element in $\g_v$.  Note that if the
reflection in the equator {\it preserves} the tessellation, then we
immediately obtain that $G_v\cong G_e \times \z_2$.  If the
reflection in the equator {\it does not preserve} the tessellation,
then we obtain that $G_v\cong G_e \rtimes \z_2$.  One can perform
the same analysis at the vertex $w$, and hence find an expression
for $\stab_{\g}(\eta)$ as an amalgamation of the groups $G_v,G_w$
over the index 2 subgroups $G_e$.

We make two observations: first of all, the stabilizer of an edge
will always be a special subgroup of $\g$, generated by a pair of
canonical generators in $\g$. In particular, the group $G_e$ will
always be a dihedral group $D_k$ for some $k$.  Now the vertex
groups are of two types: (1) if the reflection in the equator
preserves the tessellation, we obtain $G_v\cong \z_2\times D_k$, or
(2) if the reflection in the equator does not preserve the
tessellation, then one can explicitly read off the semi-direct
product structure from the tessellation, and in fact it is easy to
see that $G_v\cong D_{2k}$.  In the table below, we list out, for
each of the finite parabolic subgroups we need to consider, the
edges that reflect, as well as the corresponding $G_v$. Let us
explain the notation used in the table: the first column gives the
various finite parabolic subgroups that occur, the second column
lists the angles that appear in the spherical triangles of the
corresponding tessellation of $S_v$.  The remaining three columns
are ordered from smallest angle to largest, and expresses whether
(1) the corresponding edge extends (i.e. does not reflect) at $v$,
and (2) if it reflects, the corresponding subgroup $G_v$.

\renewcommand{\arraystretch}{1.5}
\begin{equation*}
\begin{array}{|c|c|c|c|c|}
\hline
  \z_2\times D_2 & \pi/2,\pi/2,\pi/2 & \z_2\times D_2 & \z_2\times D_2 & \z_2\times D_2 \\ \hline
  \z_2\times D_3 & \pi/3,\pi/2,\pi/2 & \z_2\times D_3 & \text{extends} & \text{extends} \\ \hline
  \z_2\times D_4 & \pi/4,\pi/2,\pi/2 & \z_2\times D_4 & \text{extends} & \text{extends} \\ \hline
  \z_2\times D_5 & \pi/5,\pi/2,\pi/2 & \z_2\times D_5 & \text{extends} & \text{extends} \\ \hline
  \z_2\times D_6 & \pi/6,\pi/2,\pi/2 & \z_2\times D_6 & \text{extends} & \text{extends} \\ \hline
  S_4 & \pi/3,\pi/3,\pi/2 & \text{extends} & \text{extends} & D_4 \\ \hline
  \z_2 \times S_4 & \pi/4,\pi/3,\pi/2 & \z_2\times D_4 & D_6 & \z_2\times D_2 \\ \hline
  \z_2 \times A_5 & \pi/5,\pi/3,\pi/2 & D_{10} & D_6  & \z_2\times D_2\\ \hline
\end{array}
\end{equation*}

\vskip 5pt

\centerline{\bf Table 1: Finite parabolic subgroups \& local
behavior of edges.}

\vskip 20pt

From the Coxeter diagrams of the 32 groups we are considering, we
can now read off quite easily the number (up to conjugacy) of
stabilizers of type I geodesics.  We now proceed to summarize the
results of this procedure, which we list out in Tables 2, 3, and 4.
We remind the reader that, in addition
to these subgroups, there will also be (up to conjugacy) countably
infinitely many maximal $\vc$ subgroup of hyperbolic type isomorphic
to one of $\z, D_\infty, \z_2\times \z, \z_2\times D_\infty$ (coming
from stabilizers of type II and type III geodesics).  The list below
can be thought of as the ``exceptional'' maximal $\vc$ subgroups of
hyperbolic type.  Indeed, as we will see in the subsequent sections,
these will be the only maximal $\vc$ subgroups of hyperbolic type
that will actually contribute to the algebraic K-theory of the
ambient groups.

\subsection{The uniform lattices.}

There are 9 hyperbolic 3-simplex groups with fundamental domain a
compact 3-simplex in $\mathbb H^3$.  The number and type of
(non-finite) stabilizers of type I geodesics are listed in the
following table:

\renewcommand{\arraystretch}{1.5}
\begin{equation*}
\begin{array}{|c|c|}
\hline \g & \stab_{\g} \\
\hline [(3^{3},5)]  & D_6*_{D_3}D_6,
(D_2\times \z_2)*_{D_2}D_4 \; (\text{twice}),  D_{10}*_{D_5}D_{10}
\\ \hline [5,3,5] & D_2\times D_\infty,  D_6*_{D_3}D_6,  (D_5\times
\z_2)*_{D_5}D_{10} \; (\text{twice}) \\
\hline [(3^{3},4)] &
D_4\times D_\infty, D_6*_{D_3}D_6,  (D_2\times \z_2)*_{D_2}D_4
\;(\text{twice}) \\
\hline [3,5,3] & D_2\times D_\infty,
(D_3\times\z_2)*_{D_3}D_{6} (\text{twice}), D_{10}*_{D_5}D_{10}\\
\hline [5,3^{1,1}] & D_2\times D_\infty \;(\text{twice}),
D_6*_{D_3}D_6, D_{10}*_{D_5}D_{10}, (D_2\times \z_2)*_{D_2}D_4\\
\hline [4,3,5] &  D_2\times D_\infty, \; (\text{twice}), D_4\times
D_\infty, D_6*_{D_3}D_6, (D_5\times \z_2)*_{D_5}D_{10} \\
\hline [(3,5)^{[2]}] & D_2\times D_\infty \; (\text{twice}), D_6*_{D_3}D_6
(\text{twice}), D_{10}*_{D_5}D_{10} (\text{twice})\\
\hline [(3,4,3,5)] & D_2\times D_\infty \; (\text{twice}), D_4\times
D_\infty , D_6*_{D_3}D_6 (\text{twice}),D_{10}*_{D_5}D_{10}\\
\hline [(3,4)^{[2]} & D_2\times D_\infty \;(\text{twice}), D_4\times
D_\infty (\text{twice}), D_6*_{D_3}D_6 (\text{twice})\\
\hline
\end{array}
\end{equation*}

\centerline{\bf Table 2: Structure of subgroups of cocompact groups.}

\vskip 20pt

\subsection{One ideal vertex.}

We have nine such Coxeter groups, namely the groups: $[5,3^{[3]}]$,
$[5,3,6]$, $[3^2,4^2]$, $[4,3^{[3]}]$, $[3,3^{[3]}]$, $[3,4^{1,1}]$,
$[4,3,6]$, $[3,3,6]$, and $[3,4,4]$.  The number and type of
(non-finite) stabilizers of type I geodesics, as well as the cusp
subgroups are listed in the following table:

\renewcommand{\arraystretch}{1.5}
\begin{equation*}
\begin{array}{|c|c|c|c|}
\hline \g & & \stab_{\g} & \text{Cusp} \\
\hline [3,3^{[3]}] & 1 \;
\text{edge} & D_4*_{D_2}D_4 & [3^{[3]}] \\
\hline [3,3,6] & 1 \;
\text{edge} & (D_2\times \z_2)*_{D_2}D_4 & [3,6] \\
\hline
[5,3^{[3]}] & 2 \; \text{edges} & D_2\times D_\infty, \quad
D_{10}*_{D_5}D_{10} & [3^{[3]}] \\
\hline [5,3,6] & 2 \;
\text{edges} & D_2\times D_\infty, \quad (D_{5}\times
\z_2)*_{D_5}D_{10} & [3,6] \\
\hline [(3^2,4^2)] & 2 \;
\text{edges} & D_2\times D_\infty, \quad \quad D_{6}*_{D_3}D_{6} &
[4,4] \\
\hline [4,3^{[3]}] & 2 \; \text{edges} & D_2\times
D_\infty, \quad \quad D_4\times D_\infty & [3^{[3]}] \\
\hline
[3,4,4] & 2 \; \text{edges} & D_2\times D_\infty, \quad \quad
(D_{3}\times \z_2)*_{D_3}D_{6} & [4,4] \\
\hline [3,4^{1,1}] & 3 \;
\text{edges} & D_2\times D_\infty (\text{twice}),\quad \quad
D_{6}*_{D_3}D_{6} & [4,4] \\
\hline [4,3,6] & 3 \; \text{edges} &
D_2\times D_\infty (\text{twice}),\quad \quad D_4\times D_\infty &
[3,6] \\ \hline
\end{array}
\end{equation*}

\vskip 5pt

\centerline{\bf Table 3: Structure of subgroups of 1-ideal vertex
groups.}

\vskip 20pt

\subsection{Two ideal vertices.}

We have nine such Coxeter groups, namely the groups: $[(3,5,3,6)]$,
$[(3,4^3)]$, $[(3,4,3,6)]$, $[(3^3,6)]$, $[3^{[3,3]}]$,
$[6,3^{1,1}]$, $[3,6,3]$, $[6,3,6]$, and $[4,4,4]$.  Note that for
these groups, we have only one edge segment in the fundamental
domain to consider.

For the groups $[(3^3,6)]$ and $[3,6,3]$, the edge extends to one of
the non-compact edges, and hence we again have that there are no
periodic geodesics of type I.  So in both of these cases, we have
that the maximal virtually infinite $\vc$ subgroups of hyperbolic type
are isomorphic to $\z$, $D_\infty$, $\z_2\times \z$, or $\z_2\times
D_\infty$.

In the remaining cases, the edge reflects at both of its endpoints.
The stabilizers we obtain, as well as the two cusp subgroups, are
listed out in the following table:

\renewcommand{\arraystretch}{1.5}
\begin{equation*}
\begin{array}{|c|c|c|}
\hline
\g  & \stab_{\g} & \text{Cusp} \\ \hline
[(3,5,3,6)] & D_{10}*_{D_5}D_{10} & [3,6] \;(\text{twice}) \\ \hline
[(3,4^3)] & D_6*_{D_3}D_6 & [4,4] \;(\text{twice}) \\ \hline
[(3,4,3,6)] & D_4\times D_\infty & [3,6] \;(\text{twice}) \\ \hline
[3^{[3,3]}] & D_4*_{D_2}D_4 & [3^{[3]}] \;(\text{twice})\\ \hline
[6,3^{1,1}] & (D_2\times \z_2)*_{D_2}D_4 & [3,6] \;(\text{twice}) \\ \hline
[6,3,6] & D_3\times D_\infty & [3,6] \;(\text{twice}) \\ \hline
[4,4,4] & D_2\times D_\infty & [4,4] \;(\text{twice}) \\ \hline
\end{array}
\end{equation*}

\vskip 5pt

\centerline{\bf Table 4: Structure of subgroups of 2-ideal vertex
groups.}

\vskip 20pt

\subsection{Three ideal vertices.}

We have two such Coxeter group: $[6,3^{[3]}]$ and $[4^{1,1,1}]$.
Again, there will be no periodic geodesics of type I.  So we obtain
that the maximal virtually infinite $\vc$ subgroups of hyperbolic type
are isomorphic to $\z$, $D_\infty$, $\z_2\times \z$, or $\z_2\times
D_\infty$.

\subsection{Four ideal vertices.}

There are three such Coxeter groups, namely the groups $[3^{[\hskip
2pt]\times [\hskip 2pt]}]$, $[4^{[4]}]$ and $[(3,6)^{[2]}]$. It is
clear that these groups have no periodic geodesics of type I, and
hence the only maximal virtually infinite $\vc$ subgroups of hyperbolic
type in both of these groups are isomorphic to $\z$, $D_\infty$,
$\z_2\times \z$, or $\z_2\times D_\infty$.

\section{The algebraic K-theory of maximal finite subgroups.}

The only (maximal) finite subgroups which occur inside the 32 groups
we are interested in are, up to isomorphism, one of the following
groups (see the previous section):

\vskip 10pt

\noindent{\bf Maximal finite subgroups:} 1, $\mathbb Z/2$, $D_n$ for
$n=2,3,4,5,6,10$ (here $D_n$ denotes the dihedral group of order
$2n$), $D_2 \times \mathbb Z/2$, $D_6 \times \mathbb Z/2$, $S_4$,
$S_4 \times \mathbb Z/2$, and $A_5 \times \mathbb Z/2$.

\vskip 10pt

Note that these groups are precisely the various finite parabolic subgroups
(in the Coxeter sense) appearing amongst the 32 Coxeter groups we are
considering.  In the spectral sequence computing the homology
$H_*^\Gamma(E_{\fin}(\Gamma);\mathbb K\mathbb Z^{-\infty})$ the
$E^2$-term is given by the algebraic $K$-groups of the finite
subgroups. The non-trivial $K$-groups are listed in Table 5 at the
end of this section.

\vspace{.3cm}

For all but {\it four} of the finite groups in our list, their lower
algebraic $K$-theory is well known. The relevant reference are
listed below for each of these groups:

\vspace{.3cm}
\begin{itemize}
\item $\mathbb Z/2$: For the negative $K$-groups, we refer the reader
to \cite{C80a}; the fact that $K_{-1}(\mathbb Z[\mathbb Z/2])=0$ can
also be found in \cite[Theorem 10.6, page 695]{Bas68}. The vanishing
of  $\tilde{K}_0$ can be found in \cite[Corollary 5.17]{CuR87}. For
information  about the vanishing of $Wh$ we refer  the reader to
\cite{O89}.

\vspace{.2cm}
\item $D_2 \times \mathbb Z/2$: For the negative $K$-groups, we refer the reader
to \cite{C80a}. The formula in Bass \cite[Chapter 12]{Bas68} shows
also that $K_{-1}(\mathbb Z[D_2 \times \mathbb Z/2])=0$.  The
vanishing of $\tilde{K}_0$ can be found in \cite{Re76}, and for  the
information concerning  $Wh$ we refer the reader to \cite{Ma78},
\cite{Ma80},  \cite{O89}.

\vspace{.2cm}
\item $D_n, \;n=2,3,4$: For the vanishing of the negative $K$-groups, we
refer the reader to \cite{C80a}. For the $\tilde{K}_0$ we refer the
reader to \cite{Re76}. In  particular, the vanishing of
$\tilde{K}_0(\mathbb ZG)$ is proven for $G=D_3$ in \cite[Theorem
8.2]{Re76} and for $G=D_4$ in \cite[Theorem 6.4]{Re76}. For
information about $Wh$ we refer the reader to \cite{Ma78},
\cite{Ma80} and \cite{O89}.

\vspace{.2cm}
\item $D_5$: As far as we know the only $K$-groups found in the literature
are $\tilde{K}_0(\mathbb ZD_5) \cong 0$ (see \cite{Re76},
\cite{EM76}) and $K_{q}(\mathbb ZD_5)\cong 0$ for all $q \leq-2$
(see \cite{C80a}). To compute $K_{-1}(\mathbb ZD_5)$, we used
results that can be found in \cite{C80a}, and \cite{C80b}, and to
compute $Wh(D_5)$, we used results that can be found in
\cite{Ma78}, \cite{Ma80} and  \cite{O89}  (see the details in the
Section 5.1).

\vspace{.2cm}
\item $D_{10}$: As far as we know the only $K$-groups found in the
literature are $\tilde{K}_0(\mathbb ZD_{10}) \cong 0$ (see
\cite{Re76}, \cite{EM76}) and $K_{q}(\mathbb ZD_{10})\cong 0$ for
all $q \leq-2$ (see \cite{C80a}). For the $K_{-1}(\mathbb ZD_{10})$,
we used the results found in \cite{C80a}, and \cite{C80b}, and for
$Wh(D_{10})$, we used results that can be found in  \cite{Ma78},
\cite{Ma80} and  \cite{O89} (see the details in the Section 5.2).

\vspace{.2cm}
\item $D_6$: See the discussion in Section 5.1. The whitehead groups
$Wh_q(D_6)$  for $q \leq 1$ can also be found in \cite[Section
3]{Pe98}, and \cite[Section 5]{Or04}.

\vspace{.2cm}
\item $D_4 \times \mathbb Z/2$: Ortiz in \cite[Section 5]{Or04} using
results from \cite{C80a} \cite{C80b}, \cite{CuR87}, \cite{O89} and
\cite{Ma06} showed that $K_{q}(\z[D_4 \times \z/2])=0$, $q\leq -1$,
$\tilde{K}_0(\mathbb Z[D_4 \times \z/2]) \cong \z/4$, and that
$Wh(D_4 \times \z/2)$ is trivial.

\vspace{.2cm}
\item $S_4$: computed in \cite{BFPP00}

\vspace{.2cm}
\item $S_4 \times \mathbb Z/2$: computed in \cite[Section 5]{Or04}.

\end{itemize}

\vspace{.2cm} For the remaining groups in our list, we detail the
computations in the next few subsections.

\subsection{The Lower algebraic $K$-theory of $D_6 \times \mathbb Z/2$}

Carter shows in \cite{C80a} that for $q \leq -2$, $K_q(\mathbb Z
F)=0$ when $F$ is a finite group. To calculate $K_{-1}(\mathbb Z[D_6
\times \mathbb Z/2])$, we use the following formula due to Carter
\cite[Theorem 3]{C80b}.

Let $G$ be a group of order $n$, let $p$ denote a prime number, let
${\hat{\mathbb Z}}_p$ denote the $p$-adic integers and let
 ${\hat{\mathbb Q}}_p$ denote the $p$-adic numbers.  Then the following
sequence is exact:
\[
0 \rightarrow K_0(\mathbb Z) \rightarrow K_0 (\mathbb Q G) \oplus
\bigoplus_{p|n}^{}\;K_0 ({\hat{\mathbb Z}}_p G) \rightarrow
\bigoplus_{p|n}^{} K_0 ({\hat{\mathbb Q}}_p G) \rightarrow
K_{-1}(\mathbb Z G) \rightarrow 0.
\]

The group algebra $\mathbb Q[D_6 \times \mathbb Z/2]$ is isomorphic
to $\mathbb Q^8 \times (M_2(\mathbb Q))^4$  and the same statement
is true if $\mathbb Q$ is replaced by ${\hat{\mathbb Q}}_2$ and
${\hat{\mathbb Q}}_3$. Hence $K_0(\mathbb Q[D_6 \times \mathbb Z/2])
\cong K_0({\hat{\mathbb Q}}_2 [D_6 \times \mathbb Z/2])\cong
K_0({\hat{\mathbb Q}}_3[D_6 \times \mathbb Z/2])\cong \mathbb
Z^{12}$. The integral $p$-adic terms are $K_0({\hat{\mathbb Z}}_2
[D_6 \times \mathbb Z/2]) \cong K_0(\mathbb F_2[D_6 \times \mathbb
Z/2] \cong K_0(\mathbb F_2[D_6]) \cong \mathbb Z^2$, (see \cite[page
350]{Or04}), and $K_0({\hat{\mathbb Z}}_3[D_6 \times \mathbb Z/2])
\cong K_0(\mathbb F_3[D_6 \times \mathbb Z/2]) \cong K_0(\mathbb
F_3[(\mathbb Z/2)^3]) \cong \mathbb Z^8$. Carter also shows in
\cite{C80a} that $K_{-1}(\mathbb Z[D_6\times \mathbb Z/2])$ is
torsion free, so counting ranks in the exact sequence, we have that
$K_{-1}(\mathbb Z[D_6 \times \mathbb Z/2]) \cong \mathbb Z^3$.

\vspace{.2cm} To compute $\tilde{K}_0(\mathbb Z[D_6 \times \mathbb
Z/2])$, consider the following Cartesian square
\[
\xymatrix@C=20pt@R=30pt{
\mathbb Z[\mathbb Z/2][D_6] \ar[d] \ar[r] & \mathbb Z[D_6]
     \ar[d] \\
\mathbb Z[D_6] \ar[r] & \mathbb F_2[D_6]}
\]
which yields the Mayer-Vietories sequence (see \cite[Theorem
49.27]{CuR87})
\begin{equation}
\begin{split}
K_1(\mathbb Z[D_6 \times \mathbb Z/2]) \rightarrow K_1(\mathbb ZD_6) &\oplus K_1(\mathbb ZD_6) \xrightarrow{\varphi}
K_1(\mathbb F_2[D_6]) \rightarrow\\
&\rightarrow \tilde{K}_0(\mathbb
Z[\mathbb Z/2][D_6])  \rightarrow \tilde{K}_0(\mathbb
ZD_6) \oplus \tilde{K}_0(\mathbb ZD_6) \rightarrow 0
\end{split}
\end{equation}
We now proceed to compute the various terms appearing in this
sequence.

We start by looking at the terms involving $\mathbb ZD_6$.  In
\cite{Re76} Reiner shows that $\tilde{K}_0(\mathbb ZD_6)$ is
trivial. $K_1(\mathbb ZD_6)$ can be computed as follows: since
$Wh(G)$ equals $K_1(\mathbb ZG)/\{\pm G^{ab}\}$, the rank of
$K_1(\mathbb ZG)$ is equal to the rank of $Wh(G)$.  But the rank of
$Wh(G)$ is $y=r - q$, where $r$ denotes the number of irreducible
real representations of $G$, and $q$ denotes the number of
irreducible rational representations of $G$. In \cite{Bas65} Bass
shows that $r$ is equal to the number of conjugacy classes of sets
$\{x, x^{-1}\}$, $x \in G$ and $q$ is the number of conjugacy
classes of cyclic subgroups of $G$ (see also \cite{Mi66}). For
$G=D_6$, a direct calculations shows that $r=q$, and hence that
$Wh(G)$ is purely torsion.  Next note that the torsion part of
$K_1(\mathbb ZG)$ is $\{\pm 1\} \oplus G^{ab} \oplus SK_1(\mathbb
ZG)$ (see \cite{W74}), and hence we have that $Wh(D_6)=SK_1(\mathbb
ZD_6)$.  Since Magurn \cite{Ma78} has shown that $SK_1(\mathbb ZG)$
is trivial, we see that $Wh(D_6)$ is trivial. Since
$(D_6)^{ab}=(\mathbb Z/2)^2$, we obtain that $K_1(\mathbb
Z[D_6])=(\mathbb Z/2)^3$.

Next we consider the remaining terms in the Mayer-Vietoris sequence.
For $G=D_6 \times \mathbb Z/2$, Magurn in \cite[Corollary 11]{Ma80}
shows that $Wh(D_6 \times \mathbb Z/2)=0$ (note that the rank of
$Wh(G \times \mathbb Z/2)$ is twice the rank of $Wh(G)$ since $r$
and $q$ get doubled, see Section 5.2 and 5.3). Since $(D_6 \times
\mathbb Z/2)^{ab}=(\mathbb Z/2)^3$, this yields $K_1(\mathbb Z[D_6
\times \mathbb Z/2])=(\mathbb Z/2)^4$.  Finally, Magurn in
\cite[Example 9]{Ma06}) shows that $K_1(\mathbb F_2[D_6])=(\mathbb
Z/2)^4$. Substituting all the known terms into the exact sequence in
(1) yields the following exact sequence:

\begin{equation}
(\mathbb Z/2)^4 \xrightarrow{\sigma} (\mathbb Z/2)^3 \oplus (\mathbb Z/2)^3
\xrightarrow{\varphi} (\mathbb Z/2)^4 \rightarrow \tilde{K}_0(\mathbb
Z[\mathbb Z/2][D_6]) \rightarrow 0.
\end{equation}

Next, we study the image of $\varphi: K_1(\mathbb ZD_6) \oplus
K_1(\mathbb ZD_6) \rightarrow K_1(\mathbb F_2[D_6])$. We claim that
$\im(\varphi)=(\mathbb Z/2)^2$. This can be seen as follows: first
$\im(\varphi)=\im(\psi)$ where $\psi:K_1(\mathbb
ZD_6)\longrightarrow K_1(\mathbb F_2[D_6])$  is induced by the
canonical ring homomorphism $\mathbb Z \rightarrow \mathbb F_2$.
Note the $K_1(\mathbb Z)$ is a direct summand of $K_1(\mathbb ZD_6)$
and isomorphic to $\mathbb Z/2$\,; but this summand goes to zero in
$K_1(\mathbb F_2[D_6])$ since it factors through the following
commutative square
\[
\xymatrix@C=20pt@R=30pt{
\mathbb Z/2 = K_1(\mathbb Z) \ar[d] \ar[r] & K_1(\mathbb F_2)= 0
                 \ar[d] \\
K_1(\mathbb ZD_6) \ar[r] & K_1(\mathbb F_2[D_6])}
\]
Since $K_1(\mathbb ZD_6)=(\mathbb Z/2)^3$, this forces
$\dim_{\mathbb F_2}(\im(\varphi))\leq 2$. Now from the exact
sequences given in (1) and (2), we have that:
$$\dim_{\mathbb
F_2}(\im(\varphi))=2\dim_{\mathbb F_2}(K_1(\mathbb Z[D_6])) -
\dim_{\mathbb F_2}(\ker(\varphi))=6-\dim_{\mathbb
F_2}(\im(\sigma)).$$
Since $\dim_{\mathbb F_2}(\im(\sigma)) \leq 4$,
we see that $\dim_{\mathbb F_2}(\im(\varphi))\geq 2$, which forces
$\im(\varphi) \cong (\mathbb Z/2)^2$. The exact sequence now yields
$\tilde{K}_0(\mathbb Z[\mathbb Z/2][D_6])\cong (\mathbb Z/2)^2$.

\subsection{The computation of the $K$-groups $K_{-1}(\mathbb ZD_5)$, and $Wh(D_5)$}

To compute $K_{-1}(\mathbb ZD_5)$, we need Carter's formula for
$K_{-1}$, \cite[Theorem 3]{C80b}, the reader is referred to Section
5.1.
\[
0 \rightarrow K_0(\mathbb Z) \rightarrow K_0 (\mathbb Q D_5)\oplus
\bigoplus_{p|n}^{}\;K_0 ({\hat{\mathbb Z}}_p D_5) \rightarrow
\bigoplus_{p|n}^{}\;K_0 ({\hat{\mathbb Q}}_p D_5) \rightarrow
K_{-1}(\mathbb Z D_5) \rightarrow 0.
\]
The group algebra $\mathbb QD_5$ is isomorphic to $\mathbb Q
\,\times\, \mathbb Q$, and the same statement is true if $\mathbb Q$
is replaced by $\hat{\mathbb Q}_2$ (recall that $\sqrt{5} \notin
\hat{\mathbb Q}_2)$. For $p=5$, the group algebra $\hat{\mathbb Q}_5
D_5 \cong (\hat{\mathbb Q}_5)^2 \times M_2(\hat{\mathbb Q}_5)$.
Hence $K_0(\hat{\mathbb Q}_2[D_5]) \cong K_0(\mathbb Q[D_5])\cong
\mathbb Z^2$, and $K_0(\hat{\mathbb Q}_5[D_5])\cong \mathbb Z^3$.
Using techniques described in \cite[Section 5]{CuR81}, we have that
$K_0({\hat{\mathbb Z}}_5[D_5])\cong K_0(\mathbb F_5[D_5])\cong
K_0(\mathbb F_5[\mathbb Z/2])=K_0(\mathbb F_5 \times \mathbb
F_5)=\mathbb Z^2$. Also $K_0({\hat{\mathbb Z}}_2[D_5])\cong
K_0(\mathbb F_2[D_5])\cong K_0(\mathbb F_2 \times M_2(\mathbb
F_2))=\mathbb Z^2$. Carter also shows in \cite{C80a} that
$K_{-1}(\mathbb Z[D_5])$ is torsion free, so counting ranks as
before, we have that $K_{-1}(\mathbb ZD_5) \cong 0$.

\vspace{.2cm} Next, we compute $Wh(D_5)$. Recall that $Wh(G)=
\mathbb Z^y \oplus SK_1(\mathbb ZG)$. Magurn in \cite{Ma78} proves
that $SK_1$ vanishes for all finite dihedral groups. The rank of the
torsion free part is $y=r-q$ (see Section 5.1). Since in $D_5=\langle r,
s \,|\, r^5= s^2=1, srs=r^{-1} \rangle$, there are three cyclic
subgroups modulo conjugacy (the trivial subgroup $\{ e\}$, $C_2$ and
$C_5$) and four conjugacy classes of sets $\{x, x^{-1}\}$
(consisting of $\{e\}$, $\{r, r^4\}$, $\{r^2, r^3\}$ and $\{s,
s\}$), we see that $r=4$ and $q=3$. This yields $Wh(D_5) \cong
\mathbb Z^{r-q} \cong \mathbb Z$.

\subsection{The computation of the $K$-groups $K_{-1}(\mathbb ZD_{10})$, and $Wh(D_{10})$}

Carter shows in \cite{C80a} that  for $q \leq -2$, $K_q(\mathbb
ZD_{10})=0$. To calculate $K_{-1}(\mathbb ZD_{10})$, again using
Carter's formula for $K_{-1}$ \cite[Theorem 3]{C80b}, we have (see
Section 5.1 and 5.2):
\[
0 \rightarrow K_0(\mathbb Z) \rightarrow K_0 (\mathbb Q D_{10})
\oplus \bigoplus_{p|n}^{}\;K_0
({\hat{\mathbb Z}}_p D_{10}) \rightarrow
\bigoplus_{p|n}^{}\;K_0 ({\hat{\mathbb Q}}_p D_{10}) \rightarrow
K_{-1}(\mathbb Z D_{10}) \rightarrow 0
\]

The group algebra $\mathbb Q[D_5 \times \mathbb Z/2]$ is isomorphic
to $\mathbb Q^4$ and the same statement is true if $\mathbb Q$ is
replaced by ${\hat{\mathbb Q}}_2$. For $p=5$, the group algebra
${\hat{\mathbb Q}}_5[D_5 \times \mathbb Z/2] \cong ({\hat{\mathbb
Q}}_5)^4 \times (M_2({\hat{\mathbb Q}}_5))^2$. Hence $K_0(\mathbb
Q[D_5 \times \mathbb Z/2]) \cong K_0({\hat{\mathbb Q}}_2 [D_5 \times
\mathbb Z/2]) \cong \mathbb Z^4$, and $K_0({\hat{\mathbb Q}}_5[D_5
\times \mathbb Z/2])\cong \mathbb Z^6$. The integral $p$-adic terms
are $K_0({\hat{\mathbb Z}}_2 [D_5 \times \mathbb Z/2]) \cong
K_0(\mathbb F_2[D_5 \times \mathbb Z/2]) \cong K_0(\mathbb F_2[D_5])
\cong \mathbb Z^2$ (see Section 5.2), and $K_0({\hat{\mathbb
Z}}_5[D_5 \times \mathbb Z/2]) \cong K_0(\mathbb F_5[D_5 \times
\mathbb Z/2]) \cong K_0(\mathbb F_5[(\mathbb Z/2)^2]) \cong \mathbb
Z^4$. Carter also shows in \cite{C80a} that $K_{-1}(\mathbb
Z[D_{10}])$ is torsion free, so counting ranks in the exact
sequence, we have that $K_{-1}(\mathbb ZD_{10}) \cong \mathbb Z$.

\vspace{.2cm} Next, we compute $Wh(D_{10})$. Magurn in \cite{Ma78}
proves that $SK_1$ vanishes for all finite dihedral groups. For
$D_{10}=\langle r, s \,|\, r^{10}= s^2=1, srs=r^{-1} \rangle$, we
have four cyclic subgroups (modulo conjugacy) of $D_{10}$ : the
trivial subgroup $\{ e\}$, $\langle s\rangle$, $\langle r^2\rangle$
and $\langle r\rangle$. On the other hand there are six conjugacy
classes of sets $\{x, x^{-1}\}$: $\{e\}$, $\{r, r^9\}$, $\{r^2,
r^8\}$, $\{r^3, r^7\}$, $\{r^4, r^6\}$, $\{r^5, r^5\}$ and $\{s,
s\}$. This gives us $r=6$, $q=4$, and hence $Wh(D_{10}) \cong
\mathbb Z^{r-q} \cong \mathbb Z^2$.

\subsection{The Lower algebraic $K$-theory of $A_5 \times \mathbb Z/2$}

Carter shows in \cite{C80a} that  for $q \leq -2$, $K_q(\mathbb Z
F)=0$ when $F$ is a finite group. To compute $Wh_q(A_5 \times
\mathbb Z/2)$ for $q \leq 1$,  we first claim that
\[
Wh_q(A_5)=
\begin{cases}
\mathbb Z & q=1 \\
0 & q=0 \\
0 & q \leq -1.
\end{cases}
\]
This can be seen as follows: by \cite[Theorem 14.6]{O89}, we have
that $SK_1(\mathbb ZA_5)=0$. The group $A_5$ has precisely five
(mutually nonisomorphic) irreducible real representation, giving
$r=5$. In $A_5$ the conjugacy classes of cyclic subgroups are
represented by the trivial subgroup $\{e\}$,
$\langle(12)(34)\rangle$, $\langle(123)\rangle$,
$\langle(12345)\rangle$ giving us that $q=4$.  This forces
$Wh(A_5)\cong \mathbb Z^{r-q}=\mathbb Z$. By \cite{EM76}, we have
that $\tilde{K}_0(\mathbb ZA_5)=0$; Dress induction as used in
\cite[Theorem 11.2]{O89} shows that $K_{-1}(\mathbb ZA_5)=0$, and by
\cite{C80a} we have that $K_q(\mathbb ZA_5)=0$ for $q \leq -2$.

\vspace{.2cm}

Now let us compute $Wh(A_5 \times \mathbb Z/2)$. Magurn in
\cite[Example 5]{Ma} shows that $SK_1(\mathbb Z[A_5 \times \mathbb
Z/2])=0$.  Since $\rank(Wh(G \times \mathbb Z/2))=2\rank(Wh(G))$ and
$Wh(A_5)\cong \z$ we get $Wh(A_5 \times \mathbb Z/2) \cong \z^2$.

\vspace{.2cm}
Next, we compute $K_{-1}(\mathbb Z[A_5 \times \mathbb Z/2])$.
Consider the following Cartesian square
\[
\xymatrix@C=20pt@R=30pt{
\mathbb Z[\mathbb Z/2][A_5] \ar[d] \ar[r] & \mathbb ZA_5
     \ar[d] \\
\mathbb ZA_5 \ar[r] & \mathbb F_2[A_5]}
\]
which yields the Mayer-Vietories sequence (see \cite[Theorem 49.27]{CuR87})
\begin{equation}
\begin{split}
\ldots \rightarrow \tilde{K}_0(\mathbb ZA_5) &\oplus
\tilde{K}_0(\mathbb ZA_5) \xrightarrow{\varphi}
\tilde{K}_0(\mathbb F_2[A_5]) \rightarrow\\
&\rightarrow K_{-1}(\mathbb Z[\mathbb Z/2][A_5])  \rightarrow
K_{-1}(\mathbb ZA_5) \oplus K_{-1}(\mathbb ZA_5) \rightarrow \cdots
\end{split}
\end{equation}
from which we first obtain $K_{-1}(\mathbb Z[A_5 \times \mathbb
Z/2]) \cong \tilde{K}_0(\mathbb F_2[A_5])$. Since $\mathbb F_2[A_5]
\cong \mathbb F_2 \times M_4(\mathbb F_2)$, we have that
$K_0(\mathbb F_2[A_5])\cong \mathbb Z^2$, from which it follows that
$K_{-1}(\mathbb Z[A_5 \times \mathbb Z/2]) \cong \mathbb Z$.

Next, we claim that $\tilde{K}_0(\mathbb Z[A_5 \times \mathbb Z/2])$
is trivial.  To see this, let $H$ be a subgroup of $G$. For any
locally free $\mathbb ZG$-module $M$ its restriction to $H$ (denoted
by $M_H$) is a locally free $\mathbb ZH$-module. The mapping defined
by $\lbrack M \rbrack \to \lbrack M_H \rbrack$ gives a homomorphism
of $\tilde{K}_0(\mathbb ZG) \to \tilde{K}_0(\mathbb ZH)$.

A group $H$ is {\it hyper-elementary} if $H$ is a semidirect product
$N \rtimes P$ of a cyclic normal subgroup $N$ and a subgroup $P$ of
prime order, where $(|N|, |P|)=1$. Denote $\mathcal H(G)$ the full
set of non-conjugate hyper-elementary subgroups of $G$. We shall
need the following result due to Swan (see \cite{Sw60}): the map
\begin{equation}
\tilde{K}_0(\mathbb ZG) \longrightarrow \prod_{H \in \mathcal H(G)} \tilde{K}_0(\mathbb ZH)
\end{equation}
is a monomorphism for any finite group G. In particular, if
$\tilde{K}_0(\mathbb ZH)=0$ for all $H \in \mathcal H$, then we
immediately obtain $\tilde{K}_0(\mathbb ZG)=0$.

To begin the proof, we first list a full set $\mathcal H(A_5)$ of
non-conjugate hyper-elementary subgroups of $A_5$: $\mathbb Z/2
\times \mathbb Z/2$, $D_3$ and $D_5$. Note that the hyper-elementary
subgroups of $G \times \mathbb Z/2$ are of the form $H$ or $H \times
\mathbb Z/2$ for $H \in \mathcal H(G)$.  In particular, the
non-conjugate hyper-elementary subgroups of $A_5 \times \mathbb Z/2$
are: $\mathbb Z/2 \times \mathbb Z/2$, $D_3$ and $D_5$,  $(\mathbb
Z/2)^3$, $D_3 \times \mathbb Z/2\cong D_6$, and $D_5 \times \mathbb
Z/2 \cong D_{10}$.

By the results already mentioned in Sections 5.1, 5.2, 5.3, we have
$\tilde{K}_0(\mathbb ZH)=0$ for all $H \in \mathcal H(A_5 \times
\mathbb Z/2)$.  The result of Swan on the injectivity of the map in
(4) immediately implies that $\tilde{K}_0(\mathbb Z[A_5 \times
\mathbb Z/2])=0$.

\renewcommand{\arraystretch}{1.6}
\begin{equation*}
\begin{array}{|c|c|c|c|}
\hline Q \in \vc & Wh_q \neq 0, \;q\leq-1 & \tilde{K}_0 \neq 0 & Wh
\neq 0 \\ \hline \hline D_5 &  & & \mathbb Z \\ \hline D_6 & K_{-1}
\cong \mathbb Z & & \\ \hline D_4 \times \mathbb Z/2 & & \mathbb Z/4
& \\ \hline D_{10} &K_{-1}\cong \mathbb Z &  & \mathbb Z^2 \\ \hline
D_6 \times \mathbb Z/2 & K_{-1} \cong \mathbb Z^3 & (\mathbb Z/2)^2&
\\ \hline S_4 \times \mathbb Z/2 &  K_{-1} \cong \mathbb Z & \mathbb
Z/4 & \\ \hline A_5 &   &  & \mathbb Z \\ \hline A_5 \times \mathbb
Z/2 &  K_{-1} \cong \mathbb Z &  & \mathbb Z^2 \\
\hline
\end{array}
\end{equation*}

\vspace{.2cm} \centerline {\bf Table 5: Lower algebraic $K$-theory
of subgroups $Q \in \fin$}

\vspace{.3cm}


\section{Cokernels of relative assembly maps for maximal infinite
virtually cyclic subgroups}

In view of Corollary 3.4, we will need for our computations the
cokernels of the relative assembly maps for the various maximal
infinite virtually cyclic subgroups of Type I.  From the tables 2,
3, 4 computed in Section 4, we have the following list containing
{\it all} the maximal infinite virtually cyclic subgroups that
appear in the 32 groups we are interested in:

\vskip 10pt

\noindent {\bf Maximal infinite virtually cyclic subgroups:}
$\mathbb Z$, $D_{\infty}$, $\mathbb Z \times \mathbb Z/2$,
$D_{\infty} \times \mathbb Z/2$, $D_n \times D_{\infty}$, for
$n=2,3,4,5$, $D_4 \ast_{D_2} D_4$, and $D_2 \times \mathbb Z/2
\ast_{D_2} D_4$.

\vskip 10pt

We first note that, for the groups in our list, the cokernels are
known to be trivial in the following cases:

\begin{itemize}
\item $\mathbb Z$: by work of Bass \cite{Bas68}.

\vspace{.2cm}
\item $D_{\infty}$: by work of Waldhausen \cite{Wd78}.

\vspace{.2cm}
\item $ \mathbb Z \times \mathbb Z/2$, and $D_{\infty} \times \mathbb Z/2$: by work of Pearson \cite[Section 2]{Pe98}.

\vspace{.2cm}
\item $D_3 \times D_{\infty}$: by work of the authors \cite[Section 4]{LO}

\end{itemize}

Finally, the authors have also shown in \cite[Section 4]{LO} that
for the group $D_2 \times D_{\infty}$, the cokernels of the relative
assembly map for $n=0,1$ are countably infinite direct sums of
$\z/2$.  The remaining four groups in our list will be discussed in
the following subsections.

Observe that by a result of Farrell and Jones \cite{FJ95},
the cokernels we are interested in $H_n^{V}(E_{\fin}(V)\rightarrow *) $
are automatically trivial for $n\leq -1$.  In particular, we only need to
focus on the cases $n=0$, and $n=1$.  These cokernels are precisely
the elusive Bass, Farrell, and Waldhausen Nil-groups.  We are able to
identify these cokernels exactly, with the exception of the case
$D_4\times D_\infty$.  For this group, we content ourselves with
summarizing what we were able to obtain in Subsection 6.4.
We summarize the non-trivial cokernels in Table 6.

\renewcommand{\arraystretch}{1.6}
\begin{equation*}
\begin{array}{|c|c|c|}
\hline V \in \vc &  H_0^{V}(E_{\fin}(V)\rightarrow
*) \neq 0 & H_1^{V}(E_{\fin}(V)\rightarrow
*)\neq 0 \\
\hline
\hline D_2 \times D_{\infty} &\bigoplus_{\infty} \mathbb Z/2 &\bigoplus_{\infty}\mathbb Z/2 \\
\hline D_4 \ast_{D_2} D_4 &\bigoplus_{\infty} \mathbb Z/2 &\bigoplus_{\infty} \mathbb Z/2 \\
\hline  (D_2\times \z/2)*_{D_2}D_4 & \bigoplus_{\infty} \mathbb Z/2 & \bigoplus_{\infty} \mathbb Z/2 \\
\hline D_4 \times D_{\infty} &  Nil_0 & Nil_1 \\
\hline
\end{array}
\end{equation*}

\vspace{.2cm} \centerline {\bf Table 6: Cokernels of relative assembly map for
maximal $V \in \vc$}

\vspace{.3cm}


\subsection{The Lower algebraic $K$-theory of $D_5 \times D_{\infty}$}

First, note that $D_5\times D_{\infty} \cong D_{10} \ast_{D_5}
D_{10}$. As before $K_n(\mathbb Z Q)$ is zero for $n <-1$ (see
\cite{FJ95}).  Since $K_{-1}(\mathbb ZD_5)=0$ (see Section 5.2), and
$K_{-1}(\mathbb ZD_{10})=\mathbb Z$ (see Section 5.3), we see that
for $Q=D_{10} \ast_{D_5} D_{10}$, we have $K_{-1}(\mathbb
ZQ)=\mathbb Z \oplus \mathbb Z$.

\vspace{.2cm} For the remaining $K$-groups, we make use of
\cite[Lemma 3.8]{CP02}.  For $Q=D_{10} \ast_{D_5} D_{10}$, $\tilde
{K}_0(\mathbb ZQ) \cong NK_0(\mathbb ZD_5; C_1, C_2)$, where
$C_i=\mathbb Z[D_{10} - D_5]$ is the $\mathbb ZD_5$-bimodule
generated by $D_{10} - D_5$ for $i=1,2$, (see Section 5.2 and 5.3
for the $\tilde{K}_0(\mathbb ZD_n)$ for $n=5,10$), and $Wh(Q) \cong
\mathbb Z^3 \oplus NK_1(\mathbb ZD_5; C_1, C_2)$, with $C_1$ and
$C_2$ as before, (see Section 5.2 and 5.3 for the $Wh(D_n)$ for
$n=5,10$). The Nil-groups appearing in these computations are the
Waldhausen's Nil-groups.

Now by \cite{LO(b)}, we know that $NK_i(\mathbb ZD_5; C_1, C_2)=0$
for $i=0,1$, vanishes if and only if the corresponding Farrell
Nil-group vanishes for the canonical index two subgroup $D_5 \times
\mathbb Z \triangleleft D_5 \times D_\infty$.  Note that in this
case, the Farrell Nil-group is untwisted, and hence is just the Bass
Nil-group $NK_i(\mathbb ZD_5)$.  But Harmon \cite{Ha87} has shown
that for finite groups $G$ of square-free order (such as $D_5$), the
Bass Nil group $NK_i(\mathbb ZG)$ vanishes for $i=0,1$. We summarize
our computations in the following:
\[
Wh_q(D_5\times D_\infty)=
\begin{cases}
\mathbb Z^3 & q=1 \\
0 & q=0 \\
\mathbb Z^2 & q= -1\\
0 & q \leq -2.
\end{cases}
\]

\subsection{The Lower algebraic $K$-theory of $D_2 \times \mathbb Z/2 \ast_{D_2} D_4$}

As before $K_n(\mathbb Z Q)$ is zero for $n <-1$ (see \cite{FJ95}).
Since $K_{-1}(\mathbb ZD_2)=0$, $K_{-1}(\mathbb Z[D_2 \times \mathbb
Z/2])=0$, and $K_{-1}(\mathbb ZD_4)=0$, we see that for $Q=D_2
\times \mathbb Z/2 \ast_{D_2} D_4$, we have that $K_{-1}(\mathbb
ZQ)= 0$.

\vspace{.2cm} For the remaining $K$-groups using \cite[Lemma
3.8]{CP02}, we have that for $Q=D_2 \times \mathbb Z/2 \ast_{D_2}
D_4$, $\tilde {K}_0(\mathbb ZQ) \cong NK_0(\mathbb ZD_2; A_1, A_2)$,
where $A_1=\mathbb Z[D_2  \times \mathbb Z/2-D_2]$ is the $\mathbb
ZD_2$ bi-module generated by $(D_2 \times \mathbb Z/2) - D_2$, and
$A_2=\mathbb Z[D_4 - D_2]$ is the $\mathbb ZD_2$ bi-module generated
by $D_4 - D_2$.  Similarly, we have that $Wh(Q) \cong NK_1(\mathbb
ZD_2; A_1, A_2)$, where $A_1$, $A_2$ are the bi-modules defined
above.

\vspace{.2cm} Now recall that in \cite[Theorem 5.2]{LO}, the authors
established that (1) $\tilde K_0(\mathbb Z[D_2 \times D_\infty])\cong
\bigoplus_{\infty} \mathbb Z/2$ and (2) $K_1(\mathbb
Z[D_2 \times D_\infty])\cong \bigoplus_{\infty} \mathbb Z/2$.  The
computation reduced to showing that the Waldhausen Nil-groups
$NK_i(\mathbb ZD_2; A_2, A_2)$ is isomorphic to an infinite
countable sum of $\mathbb Z/2$ (where the bi-module $A_2$ is defined
in the previous paragraph).  This was achieved by establishing (1)
the existence of an injection, and (2) the existence of a
(different) surjection, from the Bass Nil-group $NK_i(\mathbb D_4)
\cong \bigoplus_{\infty} \mathbb Z/2$ into the corresponding
Waldhausen Nil-group $NK_i(\mathbb ZD_4; A_2, A_2)$.  But the reader
can verify that the argument given in \cite{LO} applies verbatim to
the Waldhausen Nil-groups $NK_i(\mathbb ZD_4; A_1, A_2)$ appearing
in our present computation.

We conclude that the lower algebraic $K$-theory of $D_2 \times
\mathbb Z/2 \ast_{D_2} D_4$ is given by:
\[
Wh_q(D_2 \times \mathbb Z/2 \ast_{D_2} D_4)=
\begin{cases}
\bigoplus_{\infty} \mathbb Z/2 & q=1 \\
\bigoplus_{\infty} \mathbb Z/2 & q=0 \\
0 & q \leq -1.
\end{cases}
\]

\subsection{The Lower algebraic $K$-theory of $D_4 \ast_{D_2} D_4$}

As before $K_n(\mathbb Z Q)$ is zero for $n <-1$ (see \cite{FJ95}).
Since $K_{-1}(\mathbb ZD_2)=0$ and $K_{-1}(\mathbb ZD_4)=0$, we see
that for $Q=D_4 \ast_{D_2} D_4$, we have that $K_{-1}(\mathbb ZQ)=
0$.

\vspace{.2cm} For the remaining $K$-groups, using \cite[Lemma
3.8]{CP02}, we have that for $Q=D_4 \ast_{D_2} D_4$, $\tilde
{K}_0(\mathbb ZQ) \cong NK_0(\mathbb ZD_2; F_1, F_2)$, where for
$i=1,2$, $F_i=\mathbb Z[D_4 - D_2]$ is the $\mathbb ZD_2$ bi-module
generated by $D_4 - D_2$.  Similarly, we have that $Wh(Q) \cong
NK_1(\mathbb ZD_2; F_1, F_2)$, with $F_1$ and $F_2$ as before.

\vspace{.2cm} Now using \cite[Theorem 5.2]{LO}, we concluded that
for $Q=D_4 \ast_{D_2} D_4$
\[
Wh_q(Q)=
\begin{cases}
\bigoplus_{\infty} \mathbb Z/2 & q=1 \\
\bigoplus_{\infty} \mathbb Z/2 & q=0 \\
0 & q \leq -1.
\end{cases}
\]


\subsection{The Lower algebraic $K$-theory of $D_4 \times D_{\infty}$}

The authors were unable to obtain an explicit computation for this
group.  In this case,  we have that $D_4 \times D_{\infty} \cong
(D_4\times \z /2) *_{D_4}(D_4\times \z /2)$, and we are interested
in the Waldhausen Nil-groups associated to this splitting.  A
special case of recent independent work of several authors
(including H. Reich, F. Quinn, J. Davis and A. Ranicki) is that this
Waldhausen Nil-group is isomorphic to the Bass Nil-group associated
to the canonical index two subgroup $D_4\times \z$ inside $D_4\times
D_\infty$ (a considerable strengthening of the result of the authors
in \cite{LO(b)}). In our tables, we denote these groups by
$Nil_0=NK_0(\z D_4)$ and $Nil_1= NK_1(\z D_4)$ (the lower Nil groups
vanish by work of Farrell-Jones \cite{FJ95}). These abelian groups
are known to have the following properties:
\begin{enumerate}
\item $Nil_1$ is either trivial or infinitely generated \cite{F77},
\item $Nil_0$ is infinitely generated (see below),
\item in both of these groups, the order of every element divides $8$ (\cite{CP02},
\cite{G07}).
\end{enumerate}

It is very likely that the group $Nil_1$ is also non-trivial, but we
were unable to establish this result. In order to see that $Nil_0$
is non-trivial, consider the following Cartesian square:

\[
\xymatrix@C=20pt@R=30pt{ \mathbb Z[\z /4]\cong \z[a]/{a^4-1=0} \ar[d]
\ar[r] &  \z[a]/{a^2-1=0} \cong \mathbb Z[\z /2]
     \ar[d] \\
\mathbb Z[i] \cong \z[a]/a^2+1=0 \ar[r] & \mathbb F_2[a]/a^2-1=0
\cong \mathbb F_2[\z /2]}
\]
which yields the Cartesian square for $\z[D_4]=\z[\z /4
\rtimes_{\alpha} \z /2]=\mathbb Z[\z /4]_{\alpha}[\z /2]$:
\[
\xymatrix@C=20pt@R=30pt{ \mathbb Z[\z /4]_{\alpha}[\z /2] \ar[d] \ar[r]
& \mathbb Z[\z /2][\z /2]
     \ar[d] \\
\mathbb Z[i]_{\alpha}[\z /2] \ar[r] & \mathbb F_2[\z /2][\z /2]}
\]
where in $\mathbb Z[i]_{\alpha}[\z /2]$, the automorphism $\alpha$
acts via $\alpha(i)=-i$.  Writing $D_2=\z /2\times \z /2$ and $A=
\mathbb Z[i]_{\alpha}[\z /2]$, and applying the $NK$-functor, this
Cartesian square yields the Mayer-Vietoris sequence:
\begin{equation*}
\begin{split}
NK_2(\mathbb F_2[D_2]) \rightarrow & Nil_1 \rightarrow
NK_1(\z [D_2]) \oplus NK_{1}(A)
\rightarrow NK_{1}(\mathbb F_2[D_2]) \rightarrow \\
&\rightarrow Nil_0 \rightarrow NK_0(\mathbb Z[D_2])
\oplus NK_0(A) \rightarrow NK_{0}(\mathbb F_2[D_2])
\end{split}
\end{equation*}
Several of the groups appearing in this Mayer-Vietoris sequence are
known: the group $NK_{0}(\mathbb F_2[D_2])$ vanishes by
\cite{Bas68},
while the authors have previously shown \cite{LO} that the groups
$NK_1(\mathbb Z[D_2])$ and $NK_0(\mathbb Z[D_2])$ are likewise
countable infinite sums of $\z /2$.

Focusing on the tail end of the Mayer-Vietoris sequence, and
substituting in the expressions we already know, we see that:
$$\ldots \rightarrow Nil_0 \rightarrow NK_0(A) \oplus \bigoplus_\infty \z/2
\rightarrow 0$$ and non-triviality of $Nil_0$ follows from the
surjectivity onto the countable infinite sum of $\z /2$.  In
contrast, focusing on the head of the Mayer-Vietoris sequence, we
see that:
$$NK_2(\mathbb F_2[D_2])  \rightarrow Nil_1\rightarrow NK_1(A) \oplus \bigoplus_\infty \z/2
\rightarrow NK_1(\mathbb F_2[D_2]) \rightarrow \ldots$$ Hence to
establish that $Nil_1$ is non-trivial from this sequence, one would
need to either:
\begin{itemize}
\item establish that the first map is non-zero, i.e. understand the
map $$NK_2(\mathbb F_2[D_2]) \rightarrow NK_1(\z D_4)$$
\item establish that the second map is non-zero by showing
that the third map has a non-trivial kernel, for instance by
understanding the map $NK_1(\z [D_2])
\rightarrow NK_{1}(\mathbb F_2[D_2])$
\end{itemize}
The authors have some partial results concerning some of the terms showing
up in the head of the Mayer-Vietoris sequence, but so far have been unsuccessful
in establishing non-triviality of $Nil_1$.


\section{The spectral sequences and final computations}

We now proceed to apply Corollary 3.4 to compute the lower algebraic
K-theory of $\z \Gamma$, for $\Gamma$ one of the 32 possible
$3$-simplex hyperbolic reflection groups.  Let us recall that
Corollary 3.4 tells us that for such groups $\Gamma$, we have for
$n\leq 1$ an isomorphism:
$$K_n(\z \Gamma)\cong H_n^{\Gamma}(E_{\vc}(\Gamma);\mathbb K\mathbb Z^{-\infty})
\oplus \bigoplus_{i=1}^k H_n^{V_i}(E_{\fin}(V_i)\rightarrow
*)$$
where $\{V_i\}_{i=1}^k$ are a complete set of representatives for
the conjugacy classes of maximal infinite virtually cyclic subgroups
of Type I.

We first note that for all 32 of our groups, we have:
\begin{itemize}
\item obtained in Section 4 a complete list of the Type I maximal
infinite virtually cyclic subgroups (listed out in Tables 2, 3, and
4).
\item computed in Section 6 the groups
$$H_n^{V}(E_{\fin}(V)\rightarrow*)$$
for all the Type I maximal infinite virtually cyclic subgroups that
occur, with the exception of the case $V= D_4\times D_\infty$.
\end{itemize}
In particular, this allows us to determine the expression
$$\bigoplus_{i=1}^k H_n^{V_i}(E_{\fin}(V_i)\rightarrow*)$$
occurring in the formula above for all 32 of our groups.

Hence we are left with computing
$H_n^{\Gamma}(E_{\vc}(\Gamma);\mathbb K\mathbb Z^{-\infty})$ for
each of our 32 groups.  In order to do this, we recall that Quinn
\cite{Qu82} established the existence of a spectral sequence which
converges to this homology group, with $E^2$-terms given by:
$$E^2_{p,q}=H_p(E_{\fin}(\g)/\g\; ;\{Wh_q(\g_{\sigma})\})
\Longrightarrow Wh_{p+q}(\g).$$ The complex that gives the homology
of $E_{\vc}(\g) /\g$ with local coefficients $\{Wh_q(
\g_{\sigma})\}$ has the form
\[
\cdots \rightarrow \bigoplus_{{\sigma}^{p+1}}^{}Wh_{q}(
\g_{{\sigma}^{p+1}}) \rightarrow \bigoplus_{{\sigma}^p}^{}Wh_q(
\g_{{\sigma}^p}) \rightarrow \bigoplus_{{\sigma}^{p-1}}^{}Wh_q(
\g_{{\sigma}^{p-1}}) \cdots \rightarrow
\bigoplus_{{\sigma}^0}^{}Wh_q( \g_{{\sigma}^0}),
\]
where ${\sigma}^p$ denotes the cells in dimension $p$, and the sum
is over all $p$-dimensional cells in $E_{\vc}(\g)/\g$. The $p^{th}$
homology group of this complex will give us the entries for the
$E^2_{p,q}$-term of the spectral sequence. Let us recall that
\[
Wh_q(F)=
\begin{cases}
Wh(F), & q=1 \\
\tilde {K}_0(\mathbb Z F), & q=0 \\
K_q(\mathbb Z F), & q \leq -1.
\end{cases}
\]
Observe that for the groups we are interested it is particularly
easy to obtain a model for $E_{\fin}(\Gamma)$ with the additional
nice property that the $\Gamma$-action is cocompact.  Indeed, in the
case of the 9 cocompact lattices, one can just take the
$\Gamma$-space to be $\mathbb H^3$, with fundamental domain a
3-dimensional simplex.  In the case of the 23 non-uniform lattices,
one can $\Gamma$-equivariantly remove a disjoint collection of
horoballs from $\mathbb H^3$ to form a $\Gamma$-space $X_\Gamma$ on
which $\Gamma$ acts cocompactly.  A fundamental domain for the
$\Gamma$-action on the space $X_\Gamma$ can be obtained by (1)
taking the $3$-simplex fundamental domain $\Delta^3_\Gamma$ for the
$\Gamma$-action on $\mathbb H^3$, and (2) removing a small
neighborhood of each ideal vertex in $\Delta^3_\Gamma$.

For the resulting fundamental domain $X_\Gamma/\Gamma$, it is
particularly easy to identify the stabilizers of each cell.  Indeed,
there will always be a single $3$-dimensional cell, with trivial
isotropy.  The $2$-dimensional cells will consist of
\begin{enumerate}
\item precisely four cells corresponding to the original faces of
$\Delta^3_\Gamma$, each of which will have stabilizer $\z /2$,
\item one additional cell for each ideal vertex (obtained from
``truncating'' the vertex), with trivial stabilizer.
\end{enumerate}
Note that since $Wh_q(1)$ and $Wh_q(\z /2)$ vanish for all $q\leq
1$, this in particular implies that {\it there will never be any
contribution to the $E^2$-terms from the $3$-dimensional and
$2$-dimensional cells.}  In other words, $E^2_{p,q}=0$ except
possibly for $p=0,1$.

Now let us focus on the $1$-dimensional and $0$-dimensional cells in
the fundamental domain $X_{\g}$.  The $1$-dimensional cells will
consist of
\begin{enumerate}
\item precisely six edges, corresponding to the original edges of
$\Delta^3_\Gamma$, each of which will have stabilizer a dihedral
group $D_n$ ($n=2,3,4,5,$ or $6$).
\item three new edges for each ideal vertex (obtained from
``truncating the vertex), with stabilizer $\z/2$.
\end{enumerate}
Note that amongst these groups, the only ones that have some
non-trivial $Wh_q$ are the groups $D_5$ (for $q=1$) and $D_6$ (for
$q=-1$).  Now the $0$-dimensional cells that occur consist of
\begin{enumerate}
\item one vertex for each of the non-ideal vertices in $\Delta^3_\Gamma$.
Each of these vertices will have stabilizer a spherical Coxeter
group, isomorphic to the special subgroup of the Coxeter group
$\Gamma$ which corresponds to the vertex (up to isomorphism, these
are the groups occuring in Table 1).
\item one new vertex for each edge leading into an ideal vertex
(obtained from ``truncating the cusp'').  The stabilizer of the
vertex will coincide with the stabilizer of the corresponding edge
(and hence be a dihedral group $D_n$ where $n=2,3,4,5,$ or $6$)
\end{enumerate}
For all these groups, the non-vanishing $Wh_q$ can be found in Table
5.

Finally, we observe that since the only $1$-cells with non-trivial
$Wh_q$ are the groups $D_5$ and $D_6$, most of the morphisms in the
chain complex for the $E^2$-terms will either be zero (or in a few
cases, will clearly be isomorphisms).  The three morphisms one needs
to take care with are:
\begin{itemize}
\item $K_{-1}(\z D_6)\rightarrow K_{-1}(\z [D_6\times \z/2])$,
\item $Wh(D_5)\rightarrow Wh(D_{10})$,
\item $Wh(D_5)\rightarrow Wh(A_5 \times \z/2)$.
\end{itemize}
We proceed to analyze each of these three morphisms in the next
three sections.

\subsection{The map $K_{-1}(\z D_6)\rightarrow K_{-1}(\z [D_6\times \z/2])$.}
We start by observing that $K_{-1}(\z D_6)\cong \z$ and $K_{-1}(\z
[D_6\times \z/2]) \cong \z ^3$ (see Table 5).  We claim that the map
induced by the natural inclusion $D_6\hookrightarrow D_6\times \z
/2$ is injective, and the quotient group is isomorphic to $\z ^2$.
In order to see this, we merely note that there is a retraction from
$D_6 \times \z/2$ to the subgroup $D_6$, and hence we must have that
$K_{-1}(\z D_6)\cong \z$ is a summand inside $K_{-1}(\z [D_6\times
\z/2])\cong \z ^3$, which immediately gives our claim.

\subsection{The map $Wh(D_5)\rightarrow Wh(D_{10})$.}
We start by observing that $Wh(D_5)\cong \z$ and $Wh(D_{10}) \cong
\z ^2$ (see Table 5).  We claim that the map induced by the natural
inclusion $D_5\hookrightarrow D_{10} \cong D_5\times \z /2$ is
injective, and the quotient group is isomorphic to $\z$.  But again,
we see that there is a retraction from $D_5\times \z/2$ to the
subgroup $D_5$, and hence $Wh(D_5)\cong \z$ is a summand inside
$Wh(D_{10}) \cong \z ^2$, which gives us our claim.  Note that this
map was used implicitly in Section 6.1 (in the argument mentioned in
the second paragraph).

\subsection{The map $Wh(D_5)\rightarrow Wh(A_5\times \z /2)$.}
We start by observing that $Wh(D_5)\cong \z$ and $Wh(A_5\times \z
/2) \cong \z ^2$ (see Table 5).  We claim that the map induced by
the natural inclusion $D_5\hookrightarrow A_5\times \z /2$ is
injective, and the quotient group is isomorphic to $\z$.  Note that
in this case we do {\it not} have a retraction from the group
$A_5\times \z /2$ to the subgroup $D_5$ (since $A_5$ is simple, the
only possible non-trivial quotients would be isomorphic $\z /2$,
$A_5$, or $A_5\times \z /2$).

Let us start by observing that, from the inclusion
$D_5\hookrightarrow A_5$, we obtain that the inclusion
$D_5\hookrightarrow A_5\times \z/2$ factors through:
$$D_5\hookrightarrow D_5\times \z/2 \cong D_{10}\hookrightarrow A_5\times \z/2$$
which implies the map on Whitehead groups likewise factors through:
$$Wh(D_5) \rightarrow Wh(D_{10}) \rightarrow Wh(A_5\times \z /2).$$
Observe that the first map in the above sequence was analyzed in the
previous Section 7.2.  Furthermore the last two groups in this
sequence are abstractly isomorphic to $\z ^2$.  So in order to
obtain our claim, all we need to do is establish that the inclusion
$D_{10}\hookrightarrow A_5\times \z/2$ induces an isomorphism on
Whitehead groups.

In order to do this, we recall that Dress induction provides us with
an isomorphism (see \cite[Chapter 11]{O89}):
$$Wh(A_5\times \z/2)\cong \underrightarrow{\lim} _{H\in \mathcal H(A_5\times \z/2)} Wh(H).$$
Here $\mathcal H(A_5\times \z/2)$ consists of all hyperelementary
subgroups of $A_5\times \z/2$, the limit is over all maps induced by
inclusion and conjugation, and the isomorphism is naturally induced
by the inclusions.  Now recall (Section 5.4) that the
hyperelementary subgroups of $A_5\times \z/2$ are, up to
isomorphism: $(\z/2)^2$, $(\z /2)^3$, $D_3$, $D_5$, $D_6$, and
$D_{10}$.  Amongst these groups (see Table 5), the only groups with
non-trivial $Wh$ are the groups $D_5$ and $D_{10}$, with
$Wh(D_5)\cong \z$ and $Wh(D_{10})\cong \z ^2$.  Furthermore, inside
the group $A_5\times \z/2$, it is easy to see that:
\begin{enumerate}
\item every subgroup isomorphic to $D_5$ lies inside a subgroup isomorphic to $D_{10}$,
\item all the subgroups isomorphic to $D_{10}$ are pairwise conjugate.
\end{enumerate}
This immediately implies that the direct limit to the right is
canonically isomorphic to $Wh(D_{10})$, which gives us our desired
claim.


\renewcommand{\arraystretch}{2}
\begin{equation*}
\begin{array}{|c|c|c|c|}
\hline
\g & K_{-1}\neq 0 & \tilde{K}_0 \neq 0  & Wh \neq 0 \\

\hline

\hline [3,5,3] & \mathbb Z^4 &
\bigoplus_{\infty} \mathbb Z/2 &  \z^3 \oplus \bigoplus_{\infty} \mathbb Z/2\\

\hline [5,3,5] & \mathbb Z^4 & \bigoplus_{\infty} \mathbb Z/2
& \z^6 \oplus \bigoplus_{\infty} \mathbb Z/2 \\

\hline [(3^{3},4)] & \mathbb Z^2  &  (\mathbb Z/4)^2 \oplus
\bigoplus_{\infty} \mathbb Z/2 \oplus Nil_0
&  \bigoplus_{\infty} \mathbb Z/2\oplus Nil_1\\

\hline [5,3^{1,1}] & \mathbb Z^2 & \bigoplus_{\infty} \mathbb Z/2
&\z^3 \oplus  \bigoplus_{\infty} \mathbb Z/2\\

\hline [4,3,5] &  \mathbb Z^3 & (\mathbb Z/4)^2 \oplus
\bigoplus_{\infty} \mathbb Z/2 \oplus Nil_0
& \z^3 \oplus \bigoplus_{\infty} \mathbb Z/2 \oplus Nil_1 \\

\hline [(3^{3},5)] & \mathbb Z^2 &
\bigoplus_{\infty} \mathbb Z/2 &\z^3 \oplus \bigoplus_{\infty} \mathbb Z/2  \\

\hline [(3,5)^{[2]}] & \mathbb Z^4 &
\bigoplus_{\infty} \mathbb Z/2 &\z^6 \oplus \bigoplus_{\infty} \mathbb Z/2\\

\hline [(3,4)^{[2]}] & \mathbb Z^4& (\mathbb
Z/4)^4\oplus\bigoplus_{\infty} \mathbb Z/2 \oplus Nil_0
& \bigoplus_{\infty} \mathbb Z/2\oplus Nil_1\\

\hline [(3,4,3,5)] &  \mathbb Z^4 & (\mathbb Z/4)^2
\oplus\bigoplus_{\infty} \mathbb Z/2 \oplus Nil_0
& \z^3 \oplus \bigoplus_{\infty} \mathbb Z/2 \oplus  Nil_1 \\

\hline
\end{array}
\end{equation*}

\vspace{.3cm} \centerline{\bf Table 6: The lower algebraic
$K$-theory of the cocompact hyperbolic 3-simplex groups}


\vskip 20pt

\subsection{The spectral sequences.}

By this point of the paper we have:
\begin{itemize}
\item described a simple model for $E_{\fin}(\Gamma)$ for our
groups, and identified the stabilizers of cells (in this section)
\item computed (in Section 5) the lower algebraic $K$-groups of
the stabilizers of the cells, and
\item identified (in this section) the non-trivial morphisms
appearing in the computation of the $E^2$-terms of the Quinn
spectral sequence.
\end{itemize}
Furthermore, as explained earlier, the only possible non-zero
terms in the spectral sequence are the $E_{p,q}$ with $p=0,1$.
This boils down to understanding the homology of the complex:
\[
0 \rightarrow \bigoplus_{{\sigma}^{1}}^{}Wh_{q}( \g_{{\sigma}^{1}})
\rightarrow \bigoplus_{{\sigma}^0}^{}Wh_q( \g_{{\sigma}^0})
\rightarrow 0,
\]
But we've seen in this section that the middle map is always
injective, hence the $E_{1,q}$ terms will also vanish.  This gives
us that {\it in all 32 cases the spectral sequence collapses at the
$E^2$-term}. In fact, in all 32 cases, the only possible non-zero
$E^2$-terms are $E^2_{0,-1}$, $E^2_{0,0}$, $E^2_{0,1}$.  In
particular the $K_i(\z \g)$ vanish for $i\leq -2$.

The results obtained for $K_{-1}$, $\tilde K_0$, and $Wh$ for all 32
hyperbolic 3-simplex groups are listed out in Table 6 and Table 7.
For ease of notation, we have only entered the non-zero terms in the
Tables; all the blank squares represent entries where the
corresponding group vanishes.


\renewcommand{\arraystretch}{1.8}
\begin{equation*}
\begin{array}{|c|c|c|c|}

\hline \g & K_{-1} \neq 0 & \tilde{K}_0 \neq 0 & Wh \neq 0
\\

\hline [3^{[\hskip 2pt]\times [\hskip 2pt]}] & & & \\

\hline [4^{[4]}] &  & & \\

\hline [(3,6)^{[2]}] & \mathbb Z^2 & & \\

\hline [6,3^{[3]}] & \mathbb Z^2 & &
\\

\hline [4^{1,1,1}] & & & \\

\hline [(3,5,3,6)] & \mathbb Z^3 & & \z ^3  \\

\hline [(3,4^3)] & \mathbb Z^2& (\mathbb Z/4)^2 & \\

\hline [(3,4,3,6)] & \mathbb Z^3& (\mathbb Z/4)^2 \oplus
Nil_0 & Nil_1 \\

\hline [(3^3,6)] & \mathbb Z & & \\

\hline [3^{[3,3]}] & &
\bigoplus_{\infty} \mathbb Z/2 & \bigoplus_{\infty} \mathbb Z/2\\

\hline [6,3^{1,1}] & \mathbb Z & \bigoplus_{\infty}
\mathbb Z/2 & \bigoplus_{\infty} \mathbb Z/2  \\

\hline [3,6,3] & \mathbb Z^3& & \\

\hline [6,3,6] & \mathbb Z^{6} & (\mathbb Z/2)^4 & \\

\hline [4,4,4] & & (\mathbb Z/4)^2 \oplus
\bigoplus_{\infty} \mathbb Z/2  & \bigoplus_{\infty} \mathbb Z/2 \\

\hline [5,3^{[3]}]  & \mathbb Z^2 &  \bigoplus_{\infty} \mathbb Z/2
& \z^3 \oplus \bigoplus_{\infty} \mathbb Z/2
\\

\hline [5,3,6]  & \mathbb Z^5 & \bigoplus_{\infty} \mathbb Z/2 & \z
^3 \oplus \bigoplus_{\infty}
\mathbb Z/2 \\

\hline [(3^2,4^2)]  & \mathbb Z^2 & (\mathbb Z/4)^2
\oplus\bigoplus_{\infty} \mathbb Z/2 & \bigoplus_{\infty} \mathbb
Z/2 \\

\hline [4,3^{[3]}]  & \mathbb Z^3 & (\mathbb Z/4)^2
\oplus\bigoplus_{\infty} \mathbb Z/2 \oplus Nil_0&
\bigoplus_{\infty} \mathbb Z/2 \oplus Nil_1\\

\hline [3,3^{[3]}]  & & \bigoplus_{\infty} \mathbb Z/2 &
\bigoplus_{\infty} \mathbb Z/2 \\

\hline [3,4^{1,1}]  & \mathbb Z^2& (\mathbb Z/4)^2
\oplus\bigoplus_{\infty} \mathbb Z/2& \bigoplus_{\infty} \mathbb
Z/2\\

\hline [4,3,6] & \mathbb Z^4& (\mathbb Z/4)^2 \oplus
\bigoplus_{\infty} \mathbb Z/2 \oplus
Nil_0 & \bigoplus_{\infty} \mathbb Z/2 \oplus Nil_1 \\

\hline [3,3,6] & \mathbb Z^4 &
\bigoplus_{\infty} \mathbb Z/2  & \bigoplus_{\infty} \mathbb Z/2\\

\hline [3,4,4] & \mathbb Z^2 & (\mathbb Z/4)^2 \oplus
\bigoplus_{\infty} \mathbb Z/2 & \bigoplus_{\infty} \mathbb Z/2 \\

\hline
\end{array}
\end{equation*}

\vspace{.3cm} \centerline{\bf Table 7: The lower algebraic
$K$-theory of the non-cocompact hyperbolic 3-simplex groups}


Note that several of the group appearing in Tables 6 and 7 involve
copies of the Bass Nil-groups $NK_0(\mathbb ZD_4)$ and $NK_1(\mathbb
ZD_4)$ (see Section 5.5).  In order to simplify the notation in the
tables, we will use $Nil_0$ and $Nil_1$ to denote these two
Nil-groups. Recall that we know that the groups $Nil_0, Nil_1$ are
torsion groups, where the order of every element divides $8$, and
furthermore the group $Nil_0$ is infinitely generated (see Section
5.5).


\section{Appendix: two specific examples.}

In this Appendix we work through the entire procedure for two
specific examples (one cocompact, and one non-cocompact), with a
view of helping the reader understand the layout of the paper.

\subsection{The group $[(3,5)^{[2]}]$}
The Coxeter diagram for this group $\Gamma$ can be found in Figure
1, from which the following presentation can be read off (see
Section 2):
$$\langle w,x,y,z \hskip 5pt | \hskip 5pt w^2=x^2=y^2= z^2=1,$$
$$(wx)^3=(xy)^5=(yz)^3=(zw)^5=(wy)^2=(xz)^2 = 1\rangle.$$
This group acts on $\mathbb H^3$ cocompactly, with fundamental
domain a $3$-simplex $\Delta ^3$.  After labeling the hyperplanes
extending the four faces by the four generators of $\Gamma$, the
angles between these hyperplanes satisfy the following relationships
(see Section 2):
\begin{itemize}
\item $\angle (P_w, P_y)=\angle (P_x, P_z)= \pi /2$,
\item $\angle (P_w,P_x) = \angle (P_y, P_z) = \pi /3$,
\item $\angle (P_w, P_z)= \angle (P_x, P_y) = \pi /5$.
\end{itemize}

In particular, the action of $\Gamma$ on $\mathbb H^3$ gives a
cocompact model for $E_{\fin}(\Gamma)$, and the splitting formula
(see Corollary 3.4) tells us that we have, for all $n\leq 1$,
isomorphisms:
$$K_n(\z \Gamma)\cong H_n^{\Gamma}(E_{\fin}(\Gamma);\mathbb K\mathbb Z^{-\infty})
\oplus \bigoplus_{i=1}^k H_n^{V_i}(E_{\fin}(V_i)\rightarrow *).$$

Let us now identify the (finitely many) groups $\{V_i\}$ that appear
in the above formula.  As explained in Section 4, these groups will
arise as stabilizers of Type I geodesics, which are precisely (up to
the $\Gamma$-action) one of the six geodesics $P_w\cap P_x$, $P_w
\cap P_y$, $P_w \cap P_z$, $P_x\cap P_y$, $P_x\cap P_z$, and
$P_y\cap P_z$.  To identify the stabilizers of these geodesics, we
first need to identify the vertex stabilizers for the simplex
$\Delta ^3$.  Recall that these will be the special subgroups
generated by triples of generators.  But from the Coxeter diagram
for $\Gamma$, one immediately sees that any triple of vertices spans
out a subdiagram corresponding to the Coxeter group $[3,5]$.  This
implies that every vertex has stabilizer isomorphic to the (finite)
Coxeter group $[3,5]$, which is well known to be isomorphic to the
group $A_5\times \z/2$.  Now for each of the six type I geodesics we
have, one can consider the projection to the fundamental domain
$\Delta ^3$.  From Table 1, looking up the vertex stabilizers
$A_5\times \z/2$, we see that every one of the six geodesics
projects to precisely the associated edge in $\Delta ^3$. Now to
find the stabilizers of the geodesics, one applies Bass-Serre
theory.  The stabilizer acts on each of the geodesics with quotient
a segment, so one can write each of the stabilizers as a generalized
free product.  Furthermore, Table 1 allows us to identify the vertex
groups in the Bass-Serre graph of groups.

Let us see how this works, for instance in the case of the geodesic
$P_x \cap P_y$. The two associated hyperplanes $P_x$ and $P_y$
intersect at an angle of $\pi /5$, hence the edge group in the
Bass-Serre graph of groups will be $D_5$.  For the vertex groups, we
see that the corresponding segment in $\Delta ^3$ joins a pair of
vertices with stabilizer $A_5\times \z/2$, and correspond to the
angle of $\pi /5$ at both the vertices.  The last row in Table 1
tells us that both the vertex groups in the Bass-Serre graph of
groups will be $D_{10}$.  This tells us that the stabilizer of the
geodesic $P_x\cap P_y$ is precisely the group $D_{10}*_{D_5}D_{10}
\cong D_5\times D_{\infty}$.  Carrying this procedure out for each
of the six geodesics, one finds that the stabilizers one obtains
are:
\begin{itemize}
\item two copies of $D_{10}*_{D_5}D_{10}$, corresponding to the two
geodesics $P_x\cap P_y$ and $P_w\cap P_z$,
\item two copies of $D_6*_{D_3}D_6$, corresponding to the two
geodesics $P_w \cap P_x$ and $P_y\cap P_z$,
\item two copies of $D_2\times D_\infty$, corresponding to the two
geodesics $P_w\cap P_y$ and $P_x\cap P_z$.
\end{itemize}
Note that these are precisely the groups that are listed out in
Table 4.  Finally, amongst these six subgroups, one needs to know
which ones have a non-trivial cokernel for the relative assembly
map.  But from the work in Section 6, all the non-trivial cokernels
are listed out in Table 6.  Looking up Table 6, one sees that out of
these six groups, the only ones with non-trivial cokernels are the
two copies of $D_2\times D_\infty$, each of whom contributes
$\bigoplus _\infty \z/2$ to the $K_0(\z\Gamma)$ and $Wh(\Gamma)$.

So we are finally left with computing the homology coming from the
finite subgroups, i.e. the term
$H_n^{\Gamma}(E_{\fin}(\Gamma);\mathbb K\mathbb Z^{-\infty})$.  As
we mentioned earlier, a cocompact fundamental domain for $\mathbb
H^3/\Gamma$ is given by $\Delta ^3$.  The stabilizers of cells in
the fundamental domain can be read off from the Coxeter diagram, as
they will precisely be the special subgroups (see the discussion in
Section 7). We see that:
\begin{itemize}
\item there is one 3-dimensional cell (the interior of $\Delta ^3$),
with trivial stabilizer,
\item there are four 2-dimensional cells (the faces of $\Delta ^3$),
with stabilizer $\z /2$,
\item there are six 1-dimensional cells (the edges of $\Delta ^3$),
two of which have stabilizer $D_2$, two of which have stabilizer
$D_3$, and two of which have stabilizer $D_5$,
\item there are four 0-dimensional cells (the vertices of $\Delta
^3$), each of which has stabilizer $A_5\times \z/2$.
\end{itemize}
Now to obtain the $E^2$-terms in the Quinn spectral sequence, we
need the homology of the complex:
\[
\cdots \rightarrow \bigoplus_{{\sigma}^{p+1}}^{}Wh_{q}(
\g_{{\sigma}^{p+1}}) \rightarrow \bigoplus_{{\sigma}^p}^{}Wh_q(
\g_{{\sigma}^p}) \rightarrow \bigoplus_{{\sigma}^{p-1}}^{}Wh_q(
\g_{{\sigma}^{p-1}}) \cdots \rightarrow
\bigoplus_{{\sigma}^0}^{}Wh_q( \g_{{\sigma}^0}),
\]
where $\sigma ^p$ are the $p$-dimensional cells (which we identified
above).  But from the work in Section 5, we know explicitly all the
groups appearing in the above complex.  Indeed, looking up the
non-zero $K$-groups in Table 5, we see that for $q<-1$, the entire
complex is identically zero.  For the remaining values of $q$, we
have:

\vskip 5pt

\noindent {\bf $q=-1$:}  The complex degenerates to
$$0\rightarrow 4K_{-1}(\z[A_5\times \z/2]) \rightarrow 0,$$
where the four copies of $K_{-1}(\z[A_5\times \z/2])$ come from the
four vertices of $\Delta ^3$.  Since we know (see Table 5) that
$K_{-1}(\z[A_5\times \z/2])\cong \z$, we immediately get that
$E^2_{p,-1}$ all vanish, with the exception of $E^2_{0,-1}\cong
\z^4$.

\vskip 5pt

\noindent {\bf $q=0$:}  The complex is identically zero, and hence
we see that $E^2_{p,0}$ all vanish.

\vskip 5pt

\noindent {\bf $q=1$:}  The complex degenerates to:
$$0\rightarrow 2Wh(D_5) \rightarrow 4Wh(A_5\times \z/2) \rightarrow 0.$$
Note that the first copy of $Wh(D_5)$ comes from the edge $P_x\cap
P_y \cap \Delta ^3$, while the second copy of $Wh(D_5)$ comes from
the $P_w\cap P_z \cap \Delta ^3$.  The four copies of $Wh(A_5\times
\z/2)$ come from the four vertices of $\Delta ^3$.

Since the two edges $P_x\cap P_y \cap \Delta ^3$ and $P_w\cap P_z
\cap \Delta ^3$ are disjoint, the complex splits as a sum of two
subcomplexes, one for each of the two edges.  Focusing on the first
edge, we see that we have:
$$0\rightarrow Wh(D_5) \rightarrow 2Wh(A_5\times \z/2) \rightarrow 0$$
We know that $Wh(D_5)\cong \z$ and $Wh(A_5\times \z/2)\cong \z ^2$
(see Table 5), and that the map $Wh(D_5)\hookrightarrow Wh(A_5\times
\z/2)$ induced by inclusion is split injective (see Section 7.3).
This immediately tells us that in the chain complex above, we have
that $2Wh(A_5\times \z/2) / Wh(D_5) \cong \z ^3$. An identical
analysis for the other edge gives us that the homology of the
original complex yields $E^2_{1,1} \cong 0$ and $E^2_{0,1}\cong \z
^6$.

\vskip 5pt

Combining everything we've said so far, we see that for the Quinn
spectral sequence, the only non-zero $E^2$-terms are
$E^2_{0,-1}\cong \z^4$ and $E^2_{0,1}\cong \z ^6$.  This implies
that the spectral sequence immediately collapses, giving us that
$$H_n^{\Gamma}(E_{\fin}(\Gamma);\mathbb K\mathbb Z^{-\infty})\cong
0$$ for $n <-1, n=0$, and
$$H_{-1}^{\Gamma}(E_{\fin}(\Gamma);\mathbb K\mathbb Z^{-\infty}) \cong \z
^4,$$
$$H_1^{\Gamma}(E_{\fin}(\Gamma);\mathbb K\mathbb Z^{-\infty}) \cong \z
^6.$$

\vskip 5pt

We now have both the terms appearing in the splitting formula, and
we conclude that the lower algebraic $K$-theory of the group
$\Gamma$ is given by:
\[
Wh_n(\Gamma)=
\begin{cases}
Wh(\Gamma)\cong \z^6 \oplus \bigoplus _\infty \z/2, & n=1 \\
\tilde {K}_0(\z \Gamma) \cong \bigoplus _\infty \z/2, & n=0 \\
K_{-1}(\mathbb Z \Gamma)\cong \z ^4, & n =-1\\
K_n(\mathbb Z \Gamma)\cong 0, & n \leq -1.
\end{cases}
\]
Looking up Table 6, one finds that these are precisely the values
reported.

\subsection{The group $[3,4^{1,1}]$.}

The Coxeter diagram for this group $\Gamma$ can be found in Figure
2, from which the following presentation can be read off (see
Section 2):
$$\langle w,x,y,z \hskip 5pt | \hskip 5pt w^2=x^2=y^2= z^2=1,$$
$$(wx)^3=(xy)^4=(yz)^2=(zw)^2=(wy)^2=(xz)^4 = 1\rangle.$$
This group acts on $\mathbb H^3$ with cofinite volume, with
fundamental domain a (non-compact) $3$-simplex $\Delta ^3$ with one
ideal vertex.  After labeling the hyperplanes extending the four
faces by the four generators of $\Gamma$, the angles between these
hyperplanes satisfy the following relationships (see Section 2):
\begin{itemize}
\item $\angle (P_w, P_y)=\angle (P_w, P_z)= \angle (P_w,P_z) = \pi /2$,
\item $\angle (P_w,P_x) = \pi /3$,
\item $\angle (P_x, P_z)= \angle (P_x, P_y) = \pi /4$.
\end{itemize}
The ideal vertex arises as the intersection (at infinity) of the
three hyperplanes $P_x\cap P_y\cap P_z$, and has stabilizer the
2-dimensional crystallographic group $[4,4]$.

While the action of $\Gamma$ on $\mathbb H^3$ does not give a
cocompact model for $E_{\fin}(\Gamma)$, one can obtain such a model
by $\Gamma$-equivariantly truncating disjoint horospheres centered
at the $\Gamma$-orbits of the ideal vertex (see Section 7). The
splitting formula (see Corollary 3.4) tells us that we have, for all
$n\leq 1$, isomorphisms:
$$K_n(\z \Gamma)\cong H_n^{\Gamma}(E_{\fin}(\Gamma);\mathbb K\mathbb Z^{-\infty})
\oplus \bigoplus_{i=1}^k H_n^{V_i}(E_{\fin}(V_i)\rightarrow *).$$

Let us now identify the (finitely many) groups $\{V_i\}$ that appear
in the above formula.  As explained in Section 4, these groups will
arise as stabilizers of Type I geodesics, which are precisely (up to
the $\Gamma$-action) one of the six geodesics $P_w\cap P_x$, $P_w
\cap P_y$, $P_w \cap P_z$, $P_x\cap P_y$, $P_x\cap P_z$, and
$P_y\cap P_z$.  Note that since the geodesic segments $P_x\cap P_y$,
$P_y\cap P_z$ and $P_x \cap P_z$ project to non-compact segments in
the fundamental domain (they give rise to edges joined to the ideal
vertex), these geodesics will {\it never} have an infinite
stabilizer, and we can hence safely ignore them.

To identify the stabilizers of the remaining three geodesics, we
follow the procedure from Section 4.  We first need to identify the
(non-ideal) vertex stabilizers for the simplex $\Delta ^3$.  Recall
that these will be the special subgroups generated by triples of
generators.  But from the Coxeter diagram for $\Gamma$, one
immediately sees that the triple of vertices span out the
subdiagrams:
\begin{itemize}
\item the Coxeter group $[3,4]\cong S_4 \times \z /2$ will be the stabilizer
of the vertices $P_w\cap P_x\cap P_z$ and of the vertex $P_w\cap
P_x\cap P_y$,
\item the group $(\z/2)^3$ will be the stabilizer of the vertex $P_w\cap P_y\cap P_z$.
\end{itemize}
Now for each of the three (potentially cocompact) type I geodesics
that we have ($P_w\cap P_x$, $P_w\cap P_y$, and $P_w\cap P_z$) one
can consider the projection to the fundamental domain $\Delta ^3$.
From Table 1, looking up the vertex stabilizers $S_4\times \z/2$, we
see that every one of the three geodesics projects to precisely the
associated edge in $\Delta ^3$.

To find the stabilizers of these geodesics, we now use Bass-Serre
theory as explained in Section 4.  To find the vertex groups, one
uses Table 1, while the edge group will be precisely the dihedral
group given by the special subgroup associated to the geodesic.
This immediately gives us the stabilizers:
\begin{itemize}
\item one copy of $D_{6}*_{D_3}D_{6}$, corresponding to the
geodesic $P_w\cap P_x$,
\item two copies of $(\z/2\times D_2)*_{D_2}(\z/2\times D_2)\cong D_2\times D_\infty$,
corresponding to the two geodesics $P_w \cap P_y$ and $P_w\cap P_z$,
\end{itemize}
which are precisely the groups reported in Table 3.  Finally,
amongst these three subgroups, one needs to decide which ones have a
non-trivial cokernel for the relative assembly map.  These cokernels
are listed out in Table 6, and one sees that the only non-trivial
contribution will come from the two copies of $D_2\times D_\infty$,
each of which will contribute $\bigoplus _\infty \z/2$ to the
$\tilde K_0(\z \Gamma)$ and $Wh(\Gamma)$.

So we are finally left with computing the homology coming from the
finite subgroups, i.e. the term
$H_n^{\Gamma}(E_{\fin}(\Gamma);\mathbb K\mathbb Z^{-\infty})$.  As
we mentioned earlier, a cocompact fundamental domain for $\mathbb
H^3/\Gamma$ is given by ``truncating'' the ideal vertex from $\Delta
^3$. The stabilizers of cells in the fundamental domain can be read
off from the Coxeter diagram, as they will precisely be the special
subgroups (see the discussion in Section 7). We see that:
\begin{itemize}
\item there is one 3-dimensional cell (the interior of $\Delta ^3$),
with trivial stabilizer,
\item there are five 2-dimensional cells,
with stabilizer $\z /2$ (for the faces of the original $\Delta ^3$),
or trivial (for the face coming from truncating the ideal vertex in
$\Delta^3$),
\item there are nine 1-dimensional cells (the six edges of $\Delta ^3$,
and three edges obtained from the truncation).  Three of these will
have stabilizer $\z /2$ (those coming from truncating the ideal
vertex in $\Delta ^3$), three will have stabilizer $D_2$ (from the
edges corresponding to $P_y\cap P_z$, $P_w\cap P_y$, and $P_w\cap
P_z$), two will have stabilizer $D_4$ (from the edges corresponding
to $P_x\cap P_y$ and $P_x\cap P_z$), and one with stabilizer $D_3$
(from the edge corresponding to $P_w\cap P_x$),
\item there are six 0-dimensional cells (three non-ideal vertices of $\Delta
^3$, and three from the truncation of the ideal vertex).  Two have
stabilizers $D_4$ (from the truncation of the two edges with the
same stabilizer), one has stabilizer $D_2$ (from the truncation of
the third edge), two have stabilizer $S_4\times \z /2$ (from two of
the non-ideal vertices), and one has stabilizer $(\z /2)^3$ (from
the third non-ideal vertex).
\end{itemize}
Now to obtain the $E^2$-terms in the Quinn spectral sequence, we
need the homology of the complex:
\[
\cdots \rightarrow \bigoplus_{{\sigma}^{p+1}}^{}Wh_{q}(
\g_{{\sigma}^{p+1}}) \rightarrow \bigoplus_{{\sigma}^p}^{}Wh_q(
\g_{{\sigma}^p}) \rightarrow \bigoplus_{{\sigma}^{p-1}}^{}Wh_q(
\g_{{\sigma}^{p-1}}) \cdots \rightarrow
\bigoplus_{{\sigma}^0}^{}Wh_q( \g_{{\sigma}^0}),
\]
where $\sigma ^p$ are the $p$-dimensional cells (which we identified
above).  But from the work in Section 5, we know explicitly all the
groups appearing in the above complex.  Indeed, looking up the
non-zero $K$-groups in Table 5, we see that {\it the only one of the
cell stabilizers that has non-trivial $K$-theory is the group
$S_4\times \z /2$}.  There are two copies of this group, arising as
stabilizers of $0$-cells, and we have that $K_{-1}(\z [S_4\times
\z/2])\cong \z$ and $\tilde K_0(\z [S_4\times \z/2])\cong \z /4$.
This immediately tells us that non-zero terms in the Quinn spectral
sequence will be $E^2_{0,-1}\cong \z^2$ and $E^2_{0,0}\cong (\z
/4)^2$.  This implies that the spectral sequence immediately
collapses, giving us that
$$H_n^{\Gamma}(E_{\fin}(\Gamma);\mathbb K\mathbb Z^{-\infty})\cong
0$$ for $n <-1, n=1$, and
$$H_{-1}^{\Gamma}(E_{\fin}(\Gamma);\mathbb K\mathbb Z^{-\infty}) \cong \z
^2,$$
$$H_0^{\Gamma}(E_{\fin}(\Gamma);\mathbb K\mathbb Z^{-\infty}) \cong (\z
/4)^2.$$

\vskip 5pt

We now have both the terms appearing in the splitting formula, and
we conclude that the lower algebraic $K$-theory of the group
$\Gamma$ is given by:
\[
Wh_n(\Gamma)=
\begin{cases}
Wh(\Gamma)\cong \bigoplus _\infty \z/2, & n=1 \\
\tilde {K}_0(\z \Gamma) \cong (\z/4) ^2 \oplus \bigoplus _\infty \z/2, & n=0 \\
K_{-1}(\mathbb Z \Gamma)\cong \z ^2, & n =-1\\
K_n(\mathbb Z \Gamma)\cong 0, & n \leq -1.
\end{cases}
\]
Looking up Table 7, one finds that these are precisely the values
reported.


\end{document}